\documentclass[reqno]{amsart}
\usepackage[centertags]{amsmath}
\usepackage{graphicx}
\usepackage{amsfonts}
\usepackage{amssymb}
\usepackage{amsthm}
\usepackage[top=1in, bottom=1in, left=1in, right=1in]{geometry}

\usepackage[pdfborder={0 0 0},
   pdftitle={CREATING SUBDIVISION RULES FROM POLYHEDRA WITH IDENTIFICATIONS},
  pdfauthor={BRIAN RUSHTON}, pdftex, bookmarks=true, bookmarksnumbered = true]{hyperref}

\theoremstyle{definition}
\newtheorem{thm}{Theorem}
\theoremstyle{definition}

\theoremstyle{definition}
\newtheorem{defi}{Definition}
\theoremstyle{remark}

\theoremstyle{definition}

\newcommand{\R}{\mathbb{R}}

\newcommand{\parlengths}{\setlength{\parindent}{0pt}}
\begin{document}
\date{\today}

\pdfbookmark[1]{CREATING SUBDIVISION RULES FROM POLYHEDRA WITH IDENTIFICATIONS}{user-title-page}

\title{Creating subdivision rules from polyhedra with identifications}

\author{Brian Rushton}
\address{Department of Mathematics, Brigham Young University, Provo, UT 84602, USA}
\email{lindianr@gmail.com}

\begin{abstract}
Cannon, Swenson, and others have proved numerous theorems about subdivision rules associated to hyperbolic groups with a 2-sphere at infinity. However, few explicit examples are known. We construct an explicit subdivision rule for many 3-manifolds from polyhedral gluings. The manifolds that satisfy the conditions include all closed manifolds created from right-angled hyperbolic polyhedra, as well as many 3-manifolds with toral or hyperbolic boundary.
\end{abstract}

\maketitle\parlengths
\section{Introduction}

A long-standing conjecture of Cannon is that every Gromov hyperbolic group with a 2-sphere at infinity is a closed hyperbolic 3-manifold; in other words, it acts by isometries on hyperbolic 3-space, and its quotient is a closed manifold \cite{Combinatorial}. This is true for Gromov hyperbolic groups that are already known to be 3-manifold groups by the Geometrization conjecture, proved by Perelman \cite{Perelman}. However, it is not at all obvious that hyperbolic groups with a 2-sphere at infinity correspond to any manifold at all, and this is the reason the conjecture remains unsolved.

One approach, adopted by Cannon, Floyd, and Parry, among others, is to study subdivision rules \cite{subdivision}. All Gromov hyperbolic groups with a 2-sphere at infinity have a subdivision rule on the sphere \cite{hyperbolic}. A subdivision rule is a way of dividing the sphere into a tiling, or cell structure, with a recursive formula for dividing each tile into smaller tiles. The most famous examples include barycentric subdivision and hexagonal refinement, where a triangle is chopped up into smaller triangles, which are chopped up into smaller tiles, etc. Cannon has shown that, if a subdivision rule for a group is conformal (meaning that tiles don't get too distorted in the long run), then the group must be a hyperbolic 3-manifold group \cite{Combinatorial}. However, it has proven difficult to determine if a subdivision rule is conformal or not. Cannon and Swenson have proven the converse, i.e. that a hyperbolic 3-manifold groups have a conformal subdivision rule \cite{hyperbolic}.

One difficulty is the lack of examples. In \cite{subdivision}, Cannon, Floyd, and Parry describe a subdivision rule arising from a hyperbolic 3-orbifold with three tile types, a pentagon, a triangle, and a quadrilateral. However, no other examples have been published. On the other hand, explicit subdivision rules have been found for all non-split, prime alternating links \cite{linksubs}, most of which are finite-volume hyperbolic 3-manifold groups. However, none of these are Gromov hyperbolic.

Our goal in constructing these subdivision rules is to shed light on Cannon's conformality. Currently, the only ways of determining whether a subdivision rule is conformal or not require us to check infinitely many tilings. It seems likely that there is a simple way of telling if a subdivision rule is conformal or not just from the combinatorial structure of the tile types. By creating many examples, we can hope to explain why some tilings are conformal and others are not.

In this paper, we construct an explicit subdivision rule for a large class of 3-manifold groups, the majority of which will be hyperbolic 3-manifold groups, some closed, some finite-volume, and some of infinite volume. We start by constructing a replacement rule. A replacement rule is different from a subdivision rule. Both give a recursive way of constructing one tiling from another, but a subdivision rule requires the new tiling to include the old one as a subset, while a replacement rule does not. We then explain how to convert these replacement rules into subdivision rules in a large number of cases. As an example of our method, the tilings in Figures \ref{CircleDod} and \ref{CircleDodBig} correspond to the same manifold that Cannon, Floyd and Parry studied in \cite{subdivision}; the tiling in their paper is different from the one we obtain, but strongly related.

\begin{figure}
\begin{center}
\scalebox{.8}{\includegraphics{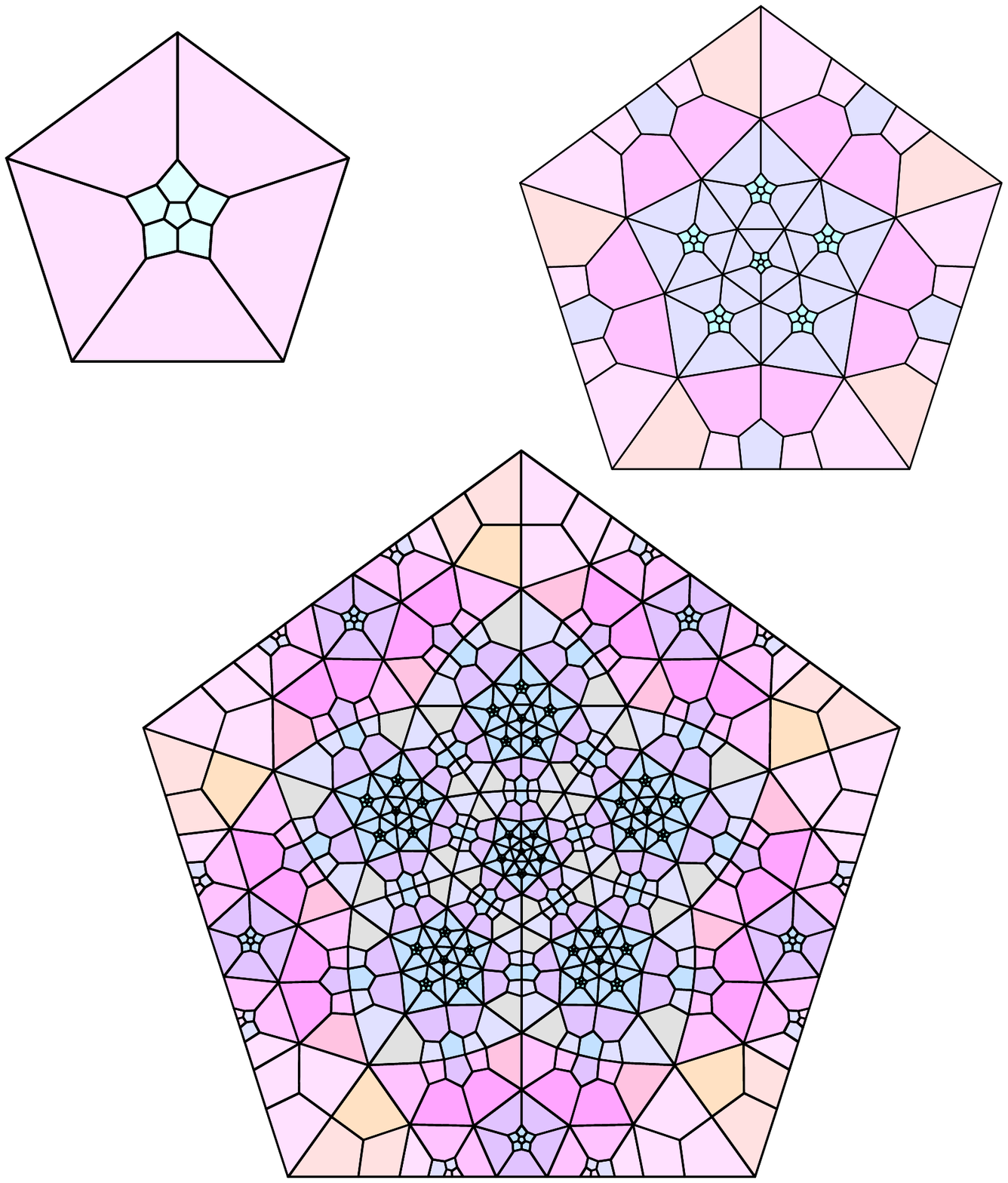}} \caption{The first three subdivisions of a tile type for a hyperbolic orbifold.}
\label{CircleDod}
\end{center}
\end{figure}

\begin{figure}
\begin{center}
\scalebox{.8}{\includegraphics{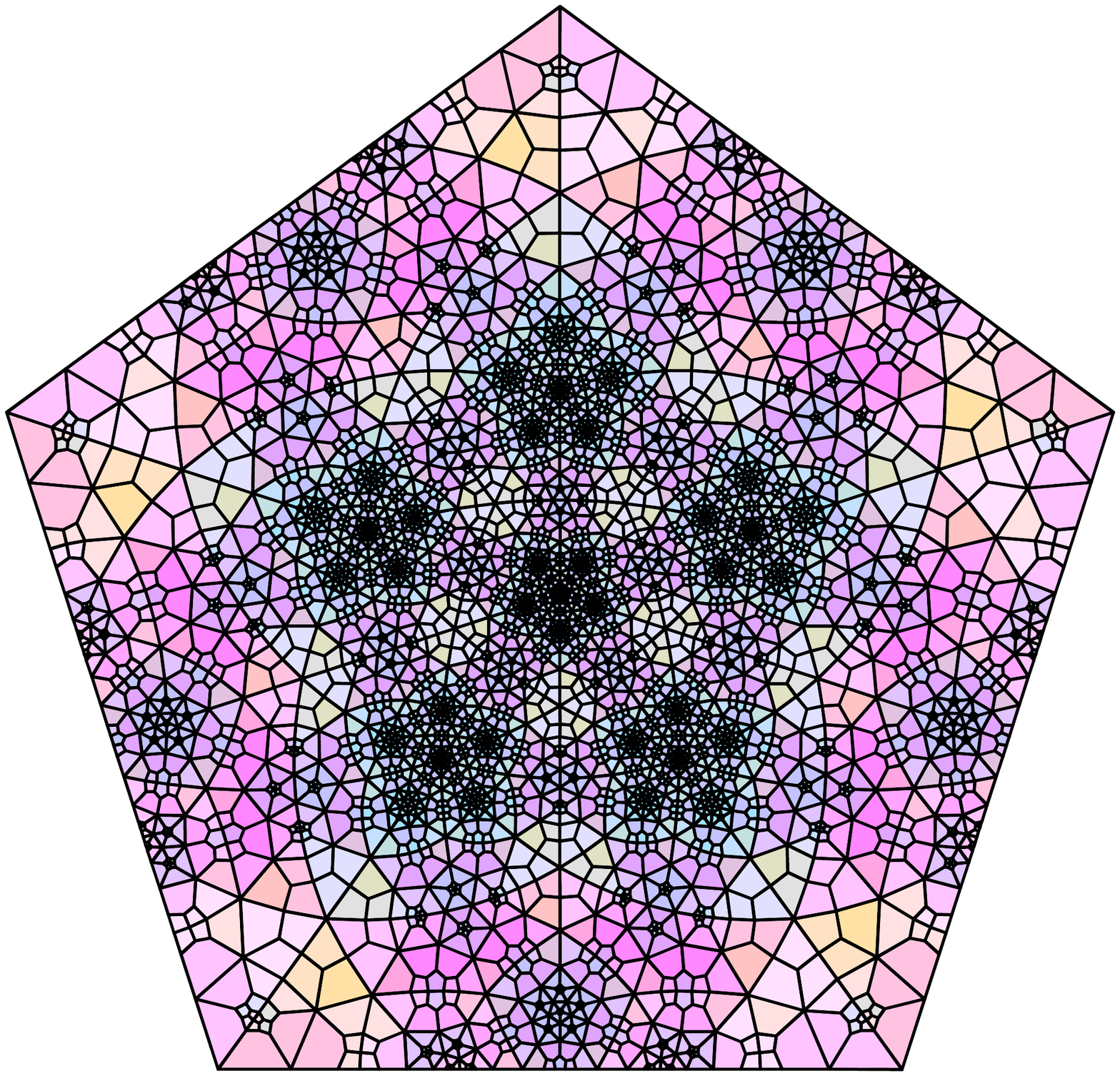}}\caption{The fourth subdivision of a tile type for the same hyperbolic orbifold.}
\label{CircleDodBig}
\end{center}
\end{figure}

\section{Preliminaries}

To associate a subdivision rule to a manifold (even a non-hyperbolic one), we create a spherical space at infinity, similar to that of word-hyperbolic groups. To do so, we approximate spheres in $\R^3$ by taking the boundary of polyhedra built up from fundamental domains.  More specifically, we let B(0) be  a single fundamental domain, and let S(0) be its boundary graph.  The structure of S(0) gives us a tiling of $S^2$.  Now, let $B(1)$ be formed from $B(0)$ by attaching polyhedra to all its exposed faces, and let $S(1)$ be its boundary, and so on.  This defines a sequence of tilings of the sphere, which defines a combinatorial structure on the space at infinity. Note $S(n)$ must be a topological sphere for all $n$.

We now define the various terms we use in this paper.

\begin{defi}\label{PolyDef} We will, for convenience, define a \textbf{polyhedron} as a topological ball with a cell structure on the boundary, where all vertices have valence three or more, all faces have at least three edges, and any two faces that intersect do so in a single edge or a vertex.
\end{defi}

\begin{defi} A \textbf{fan} is a chain $a_1,a_2,...,a_n$ of faces in a polyhedron surrounding a single vertex $v$, where $a_i\cap a_{i+1}$ contains an edge coming out of $v$ for $1\leq i \leq n+1$. If $a_1\cap a_n$ does not contain an edge coming out of $v$, then $a_1$ and $a_n$ are called the ends of the fan. If $a_1\cap a_n$ does contain such an edge, then the fan $a_1,...,a_n$ is the star of the vertex. We consider a single face to be a fan of size 1. See Figure \ref{Fans} for examples of fans.
\end{defi}

\begin{figure}
\begin{center}
\scalebox{.8}{\includegraphics{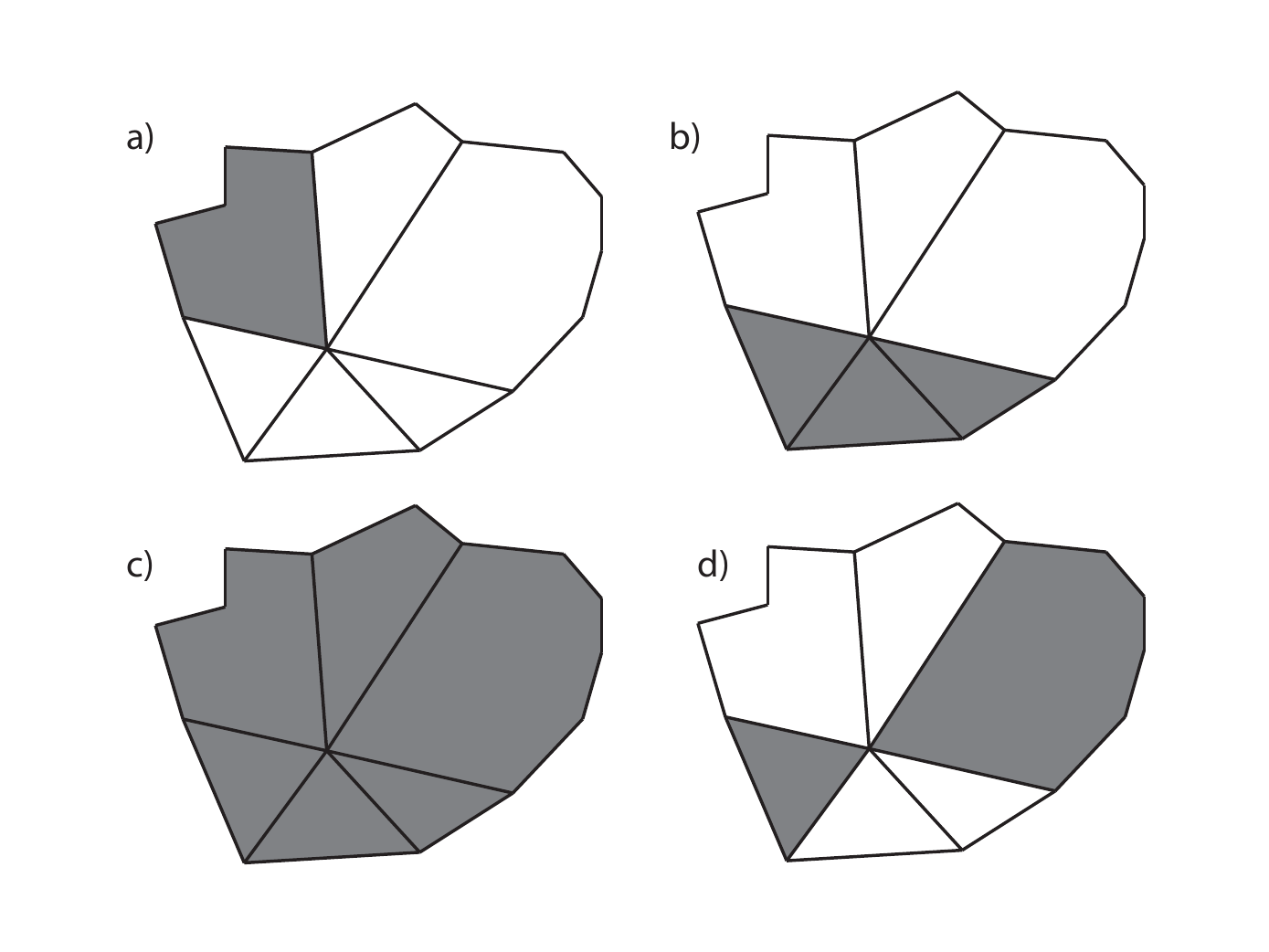}} \caption[Examples of fans.]{The shaded portions in a,b, and c are all examples of fans, while the shaded portion in d is not.}
\label{Fans}
\end{center}
\end{figure}

\begin{defi}\label{SpreadOut}
A polyhedron is \textbf{spread out} if, given any face $A$ and any fan $B$ not containing $A$, $A\cap B$ is a vertex, one edge, or two contiguous edges.
\end{defi}

Among the platonic solids, the cube, dodecahedron, octahedron, and icosahedron are spread out. The tetrahedron is not; if we take our fan to be the star of a vertex, all three of its edges intersect the remaining face.

This is analogous to alternating links. Alternating links have a well-known checkerboard polyhedral decomposition, where the boundary is an alternating diagram for the link. All hyperbolic alternating links have an alternating diagram that is reduced, non-split, and prime, and these properties are very similar to the spread-out condition above.

Andreev's theorem \cite{Andreev} (later extended by Rivin and others \cite{Rivin}) implies that all convex hyperbolic polyhedra with dihedral angles equal to $\frac{\pi}{2}$ at every edge is spread out.

We will need restrictions on edge cycle lengths.

\begin{defi}When polyhedra are glued together, the edges are identified under the gluing maps. For a given edge $e$, the \textbf{edge cycle length} of $e$ is the size of the equivalence class containing $e$.
\end{defi}

In the universal cover of a manifold, the edge cycle length turns out to be the number of distinct fundamental domains surrounding lifts of the edge. Thus, if a polyhedral fundamental domain for a hyperbolic 3-manifold has dihedral angle $\frac{2\pi}{n}$ at an edge $e$, the edge cycle length of $e$ is $n$. We are now ready to state the theorem.

\section{Existence of replacement rules}

\begin{thm}\label{BigTheorem}
Let $M$ be a manifold (possibly with boundary) that can be decomposed into polyhedra $P_1, P_2,...,P_n$. If each polyhedron is spread out, and all edge cycles have even length $\geq 4$, there exists a replacement rule for $M$. In this replacement rule, each new polyhedron is placed on exactly one fan in the previous stage of the replacement rule.
\end{thm}

\begin{proof}
Let $S(n)$ represent the tiling given by gluing $n$ layers of polyhedra to a base polyhedron. Note that in $S(0)$, all exposed faces are fans consisting of single faces. Also, all edges border fans of size 1.

As we glue on more and more polyhedra, each edge will be buried. This happens when we glue a number of polyhedra adjacent to it equal to its cycle length. All edges in $S(0)$ are adjacent to only one polyhedron; we say that these are \textbf{unburdened} edges. As we place polyhedra on the faces to either side, we get closer to closing up the cycle of polyhedra around the edge; we say this edge is now a \textbf{burdened} edge. More precisely, a burdened edge is any edge adjacent to more than one polyhedron. Because the edge cycle has even length and all edges start out adjacent to a single polyhedron, when we glue a polyhedron on either side of the edge at each stage, it will eventually only have one polyhedron remaining to glue on. If we continued to glue polyhedra on either side, we would exceed the cycle length; thus, we glue a single polyhedron onto both faces neighboring this edge. Because this gluing behavior is different from that of previous stages, we say the edge is a \textbf{loaded} edge. That is, a loaded edge needs only one polyhedron to complete its cycle, and we focus on this type of edge in our proof. All edges start out as unburdened, become burdened, become loaded, and then disappear as we glue a polyhedron over them. Every loaded edge is burdened (because it has more than one adjacent polyhedron).

Our goal is to show that, at every stage, every polyhedron is glued onto a single face or a larger fan. Because there are finitely many combinatorial types of fans, this will show that there is a recursive combinatorial way of constructing $S(n+1)$ from $S(n)$. Now, when a polyhedron is glued onto a face in $S(n)$, if that face has any loaded edges, the polyhedron will be glued onto those loaded edges and onto all faces on the other side of those loaded edges; and if those faces have loaded edges, the polyhedron is glued onto those as well, and so on. To show that all polyhedra are glued onto fans and only fans, we have to show that faces with loaded edges meet up as fans.

What do fans look like? Locally, at the vertex, they look like pizzas with some number of neighboring pieces missing (so that what remains is the sector of a circle). Each piece of pizza touches exactly two neighboring pieces, except for the ends of the pizza, which touch exactly one other piece of the fan.

Therefore, if a face has exactly two loaded edges that share a vertex, we call it a \textbf{loaded wedge}. Loaded wedges may group together into fans, where the ends of the fan (if there are any) have only one loaded edge. In this case, the ends are called \textbf{half-loaded wedges} and the entire fan is a \textbf{loaded fan}.  Our goal, then, is to show that all loaded faces can be grouped into loaded fans, and the first step in proving that is to prove that all loaded faces are loaded and half-loaded wedges. See Figure \ref{LoadedFan} for an example.

\begin{figure}
\begin{center}
\scalebox{.8}{\includegraphics{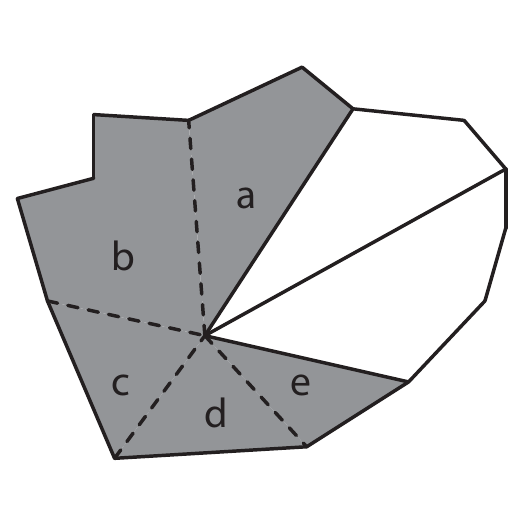}} \caption[Loaded fans and loaded wedges.]{If the shaded portion is a loaded fan, then faces b,c, and d will be examples of loaded wedges, while faces a and e will be examples of half-loaded wedges. Dashed edges are loaded, and others are not.}
\label{LoadedFan}
\end{center}
\end{figure}

As for the second step, note that we said loaded wedges \emph{may} group together into fans. Loaded wedges are, again, like pizza slices. If we rotate a pizza slice $180^o$ before putting it down, we get an unusual parallelogram of pizza instead of a sector of a circle. This shows that it is not sufficient to prove that all loaded faces are loaded or half-loaded wedges; we must also prove that they meet up in the correct way. An edge is \textbf{skew} if there are two loaded edges that intersect it, one at each vertex, and these loaded edges do not border the same face. See Figure \ref{SkewEdge} for an example of a skew edge.

\begin{figure}
\begin{center}
\scalebox{.6}{\includegraphics{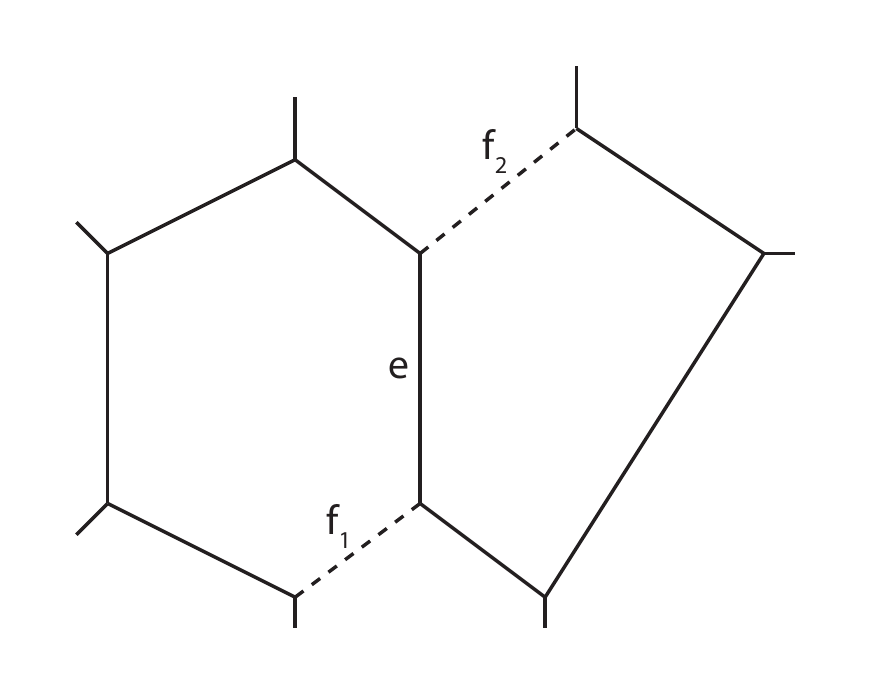}} \caption{A skew edge $e$. Here $f_1$ and $f_2$ are loaded, intersect $e$ in distinct vertices, but border different faces.}
\label{SkewEdge}
\end{center}
\end{figure}

Skew edges represent the rotated pizza slices that meet up in the wrong way. The rest of this proof will be devoted to showing that all loaded faces are loaded or half-loaded wedges, and that there are no skew edges. This will show that all loaded faces group together into fans; thus, every polyhedron that is glued on to $S(n)$ is glued on to a loaded fan, and since there are only finitely many fans on each polyhedron, this implies that there are only finitely many combinatorial types of gluings.

We proceed by induction. In $S(0)$, all edges are unloaded, and so there are no skew edges or loaded faces. Assume that $S(n)$ has no skew edges and that all faces with loaded edges are loaded or half-loaded wedges; this implies that all faces are grouped into fans. Then we form $S(n+1)$ by placing polyhedra onto fans. A fan $A$ in $S(n)$ will be replaced by several subtiles in $S(n+1)$ corresponding to unglued faces of the polyhedron. Those in the interior correspond to faces on the polyhedron that don't touch the fan $A$, and are formed of unburdened edges. Those touching a former edge of $A$ correspond to faces on the polyhedron that share an edge with the fan $A$ (where we are now considering $A$ as a subset of the polyhedron). Because our polyhedra are spread out (recall Definition \ref{SpreadOut}), these subtiles share either one edge or two contiguous edges with the fan. These edges are the only edges that are burdened, because they are the only edges of the polyhedron we are gluing on that will be adjacent to more than one polyhedron.

So, all faces in $S(n+1)$ have zero, one, or two contiguous burdened edges. This implies that if these faces have any loaded edges (which are a kind of burdened edge), the faces must be either loaded or half-loaded wedges. Thus, we have completed half of our induction.

Now, we must show that there are no skew edges. We will first show that there are no burdened skew edges.

If there is a burdened skew edge $e$ in $S(n+1)$, then by definition of skew edge (recall Figure \ref{SkewEdge}), the faces to either side each have two contiguous burdened edges. As we discussed earlier, faces with burdened edges in $S(n+1)$ correspond to faces in the new polyhedra that are adjacent to the fan we are gluing on to. Thus, if a face has two burdened edges, it must have been adjacent to a fan in two edges. Now, if a fan consists of a single face, then no other face can intersect it in two edges, by our definition of polyhedron (see Definition \ref{PolyDef}).

Thus, if a face in $S(n)$ has two burdened edges, it comes from a polyhedron glued onto a loaded fan in $S(n-1)$. Also, the two burdened edges in $S(n)$ must have touched two different faces of the loaded fan in $S(n-1)$, again by Definition \ref{PolyDef}. Moreover, the vertex they share must also have been the center vertex of the loaded fan in $S(n-1)$.

So, a burdened skew edge $e$ in $S(n)$ must come from an edge in $S(n-1)$ that is adjacent to two loaded fans, and the loaded edges of those loaded fans intersect $e$ in distinct edges. So $e$ must have been a skew edge (burdened or unburdened) in $S(n-1)$, which is impossible by our induction hypothesis. See Figure \ref{SkewInduction}.

\begin{figure}
\begin{center}
\scalebox{.8}{\includegraphics{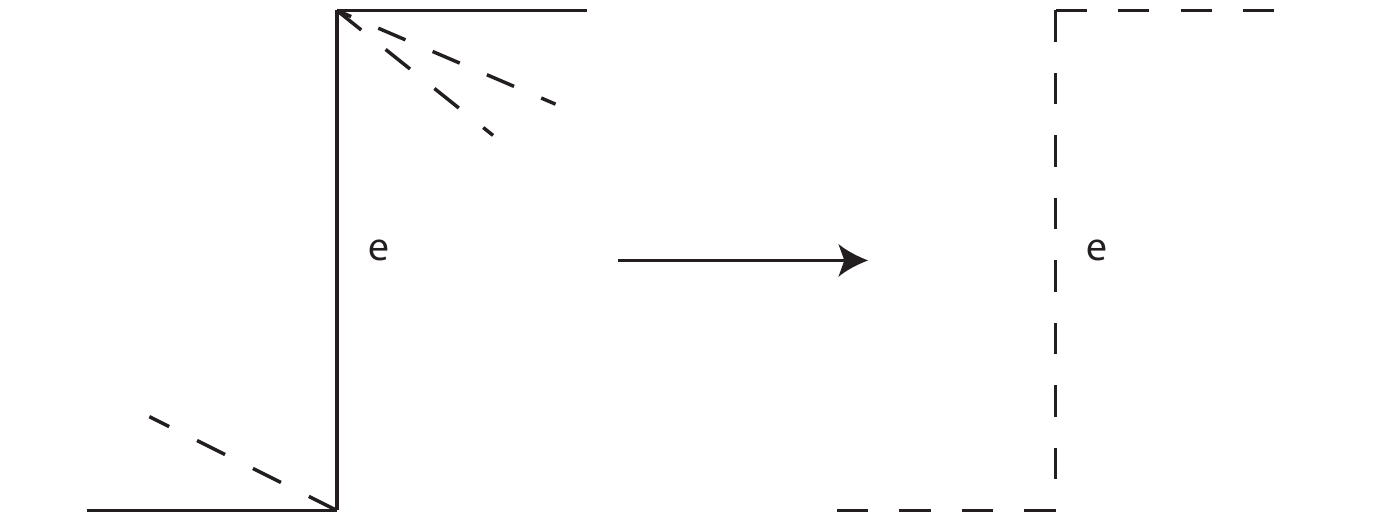}} \caption[There are no skew edges by induction.]{Burdened skew edges must come from other skew edges.}
\label{SkewInduction}
\end{center}
\end{figure}

We now show (by considering a possible counterexample) that there are no unburdened skew edges in $S(n+1)$, which completes the proof. Assume that there is an unburdened skew edge $e$. Because $e$ is unburdened, it must be in the interior of some polyhedron that was glued on to a fan $A$ in $S(n+1)$. See Figure \ref{UnburdenedSkew}. Thus, the faces on either side (call them $B$ and $C$) must also be part of the same polyhedron, and both border the fan $A$ (considered now as a subset of the polyhedron being glued on). But then $B$ intersects $A$ in an edge contiguous with $e$ and in both vertices of the interior edge $e$. Because the intersection of $A$ and $B$ is connected and is one or two edges, $B$ must be a bigon (which can't happen) or a triangle. The same applies for $C$. Thus, our polyhedra consists of a fan (namely, $A$) and two triangles sharing an edge (namely, $B$ and $C$). There can be at most four faces in the fan, as each must intersect something not in the fan (because at most two edges of each face in a fan touch other faces in the fan, and each face is a triangle or larger), and there are only four edges of $B\cup C$ intersecting the fan. If the fan is not the star of a vertex, we can add $B$ or $C$ to it, and this larger fan intersects the remaining triangle in all of its edges, a contradiction. So the fan must be the star of a vertex of valence 3 or 4. If it is of valence 3, then the star of a vertex of $e$ contains 4 tiles, and the remaining tile must intersect this star in more than 2 edges, which cannot occur. If it is of valence 4, then we can pick one vertex of $e$ and consider the fan around it consisting of $B$, $C$, and one tile of $A$. Then one of the tiles of $A$ touching the other vertex of $e$ has disconnected intersection with this fan (see Figure \ref{ContraSquares}). This is a contradiction. Thus, there are no skew edges in $S(n+1)$.

\begin{figure}
\begin{center}
\scalebox{.8}{\includegraphics{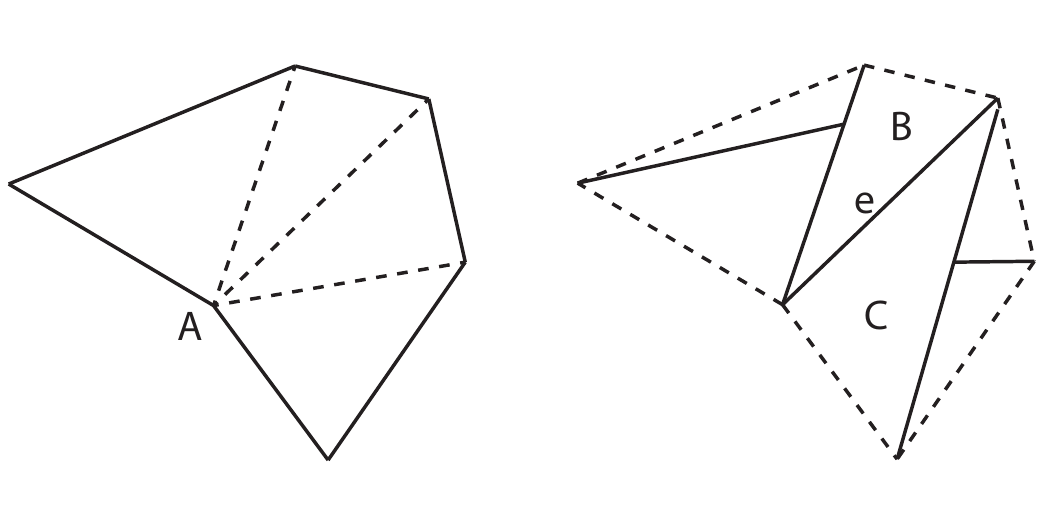}} \caption{An unburdened skew edge must look something like this.}
\label{UnburdenedSkew}
\end{center}
\end{figure}

\begin{figure}
\begin{center}
\scalebox{.9}{\includegraphics{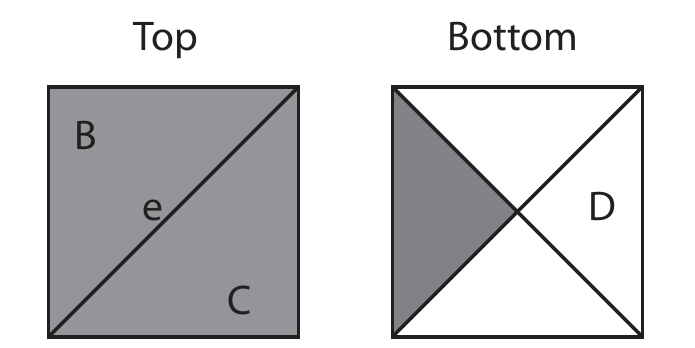}} \caption[A possible counterexample.]{Note that the fan consisting of the shaded gray region has disconnected intersection with the face D.}
\label{ContraSquares}
\end{center}
\end{figure}

Because there are no skew edges, all faces group together in fans. Since there are finitely many types of fans, we get a finite recursive algorithm for constructing the universal cover, i.e. a replacement rule.
\end{proof}

\section{Forming subdivision rules from replacement rules}

Note that Theorem \ref{BigTheorem} does not provide a way for creating a subdivision rule. However, in many cases, a subdivision rule can be created by adding extra lines. To illustrate, consider a non-hyperbolic example, the 3-dimensional torus. It has a decomposition into a single polyhedron, a cube with edge cycle lengths all equal to 4, and so has a replacement rule as described by the theorem. The three replacement types are shown in Figures \ref{TorusSquareSub}, \ref{TorusPairSub}, and \ref{TorusTripleSub}. They correspond to an unloaded face, a \textbf{loaded pair} of faces (i.e. a fan of two faces, sharing a common loaded edge), and a \textbf{loaded star} (i.e. a maximal loaded fan, containing the entire star of a vertex). By symmetry, these are the only types.

\begin{figure}
\begin{center}
\scalebox{.3}{\includegraphics{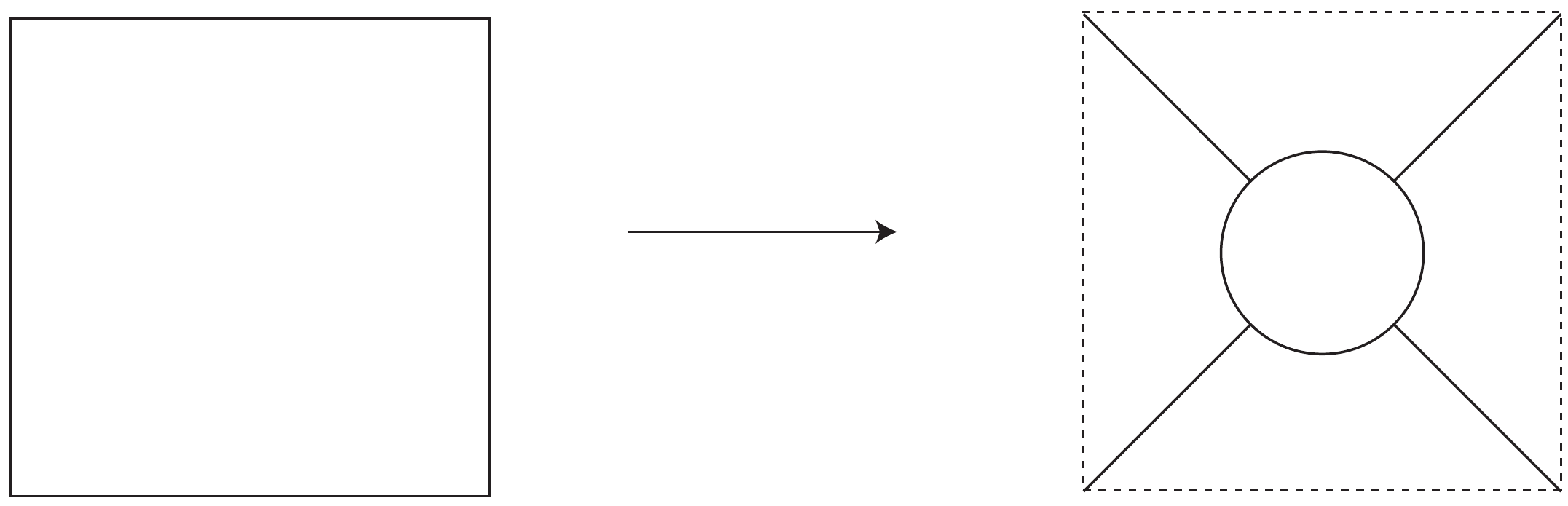}}
\caption{The replacement rule for an unloaded face.}\label{TorusSquareSub}
\end{center}
\end{figure}

\begin{figure}
\begin{center}
\scalebox{.3}{\includegraphics{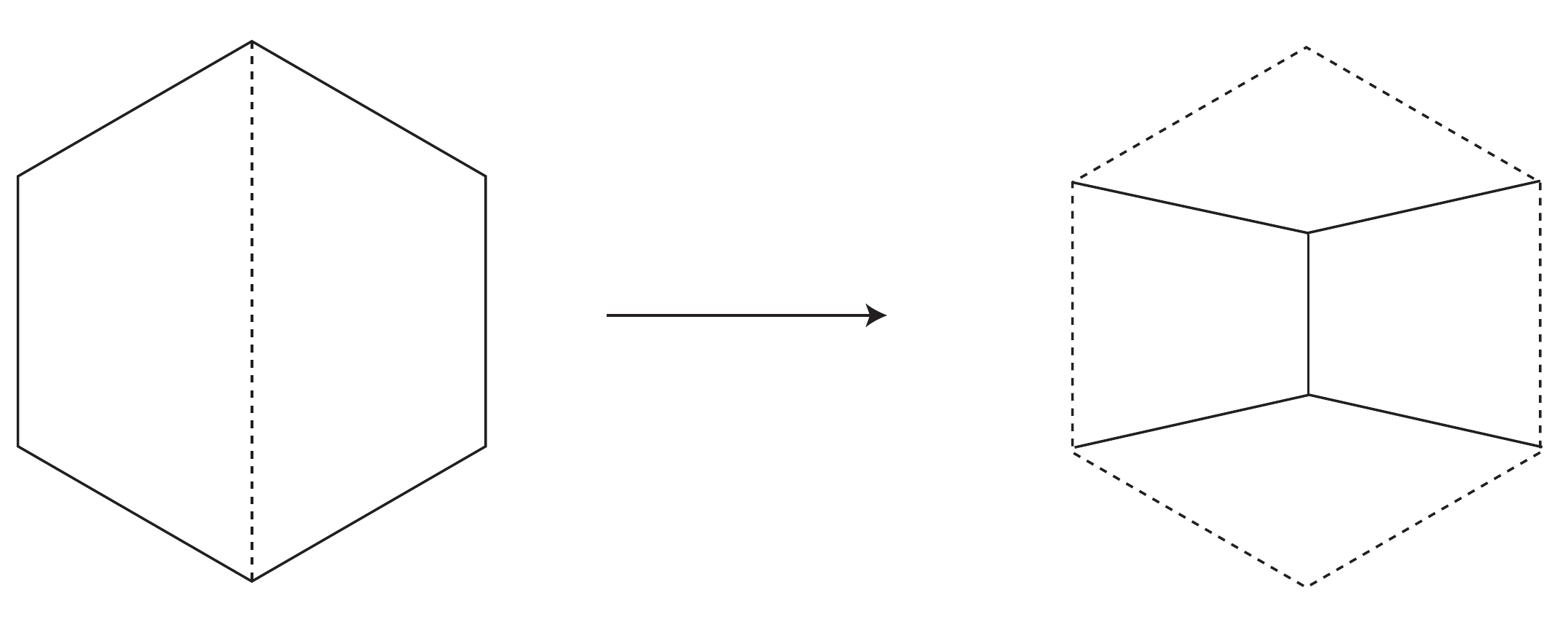}}
\caption{The replacement rule for a loaded pair.}\label{TorusPairSub}
\end{center}
\end{figure}

\begin{figure}
\begin{center}
\scalebox{.3}{\includegraphics{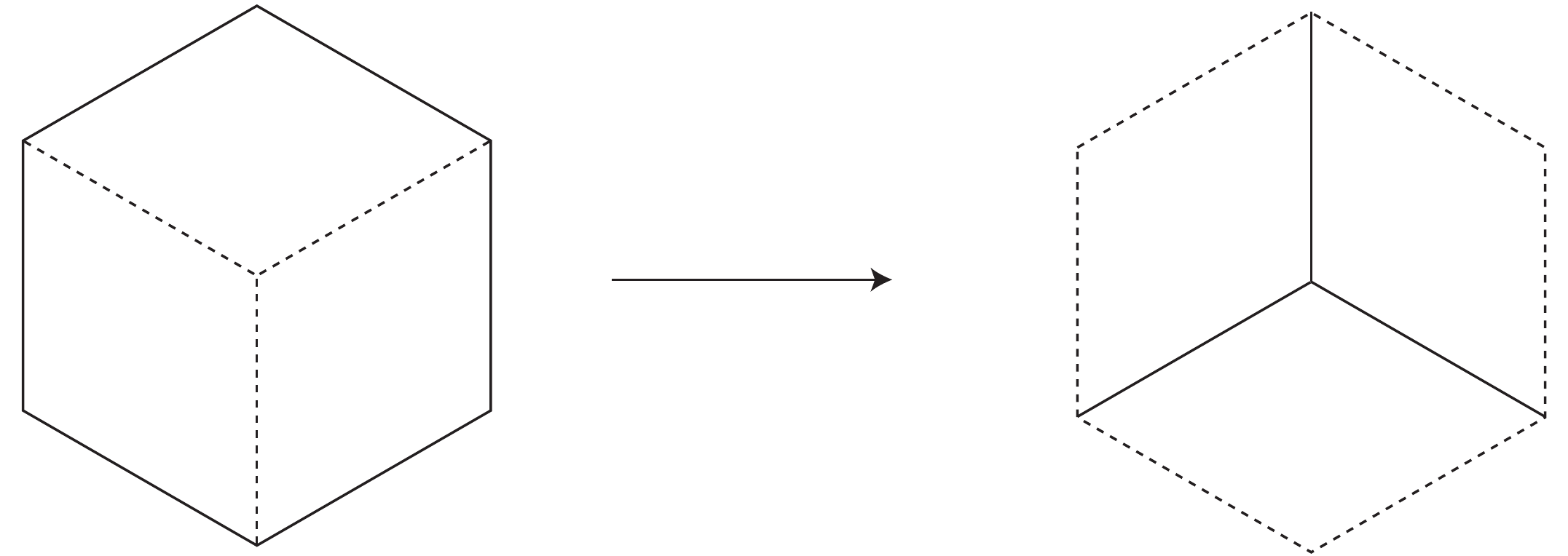}}
\caption{The replacement rule for a loaded fan.}\label{TorusTripleSub}
\end{center}
\end{figure}

Notice that this is only a replacement rule; edges are created, disappear, reappear,etc. This can be turned into a subdivision rule by a method that turns out to be very useful: adding new edges at every stage. For instance, in Figure \ref{TorusPairSub}, the center line between two squares disappears when we glue on the new cube. However, if we add a line to the new cell structure (as shown in the top half of Figure \ref{TorusAddedLines}), then the new cell structure contains the old cell structure as a subset. Thus, we have a subdivision rule.

\begin{figure}
\begin{center}
\scalebox{.6}{\includegraphics{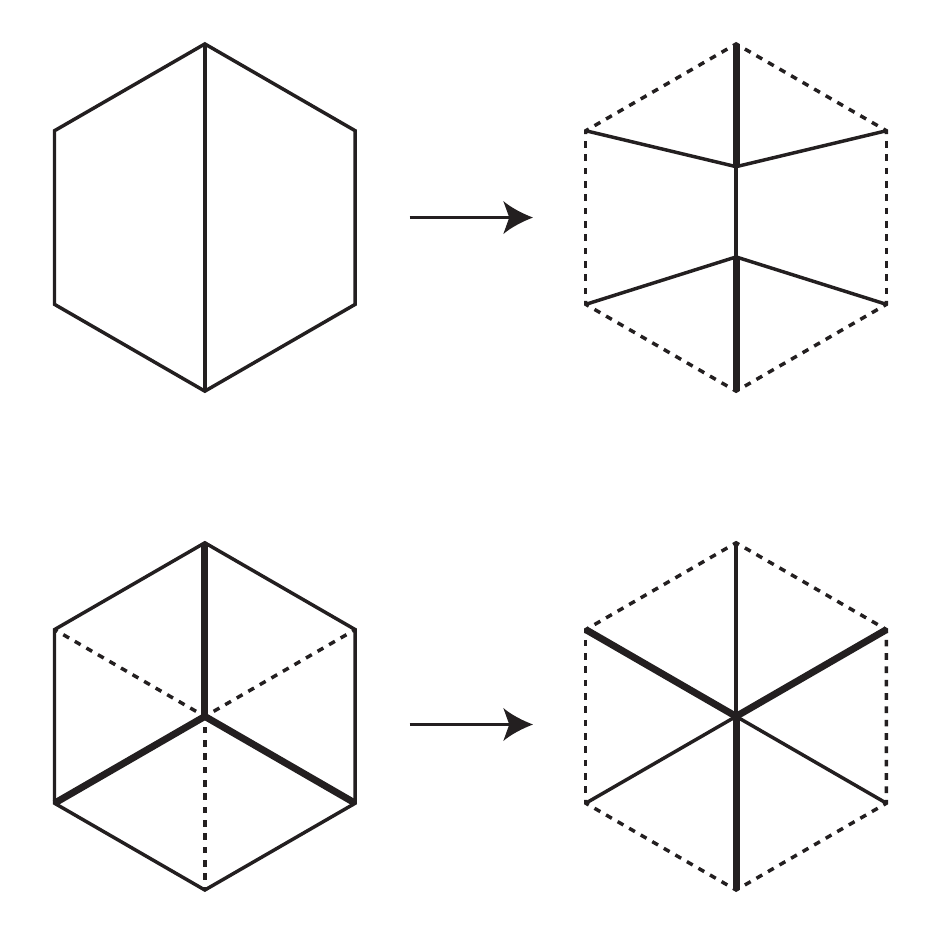}}
\caption{We can add lines to the two `loaded' tile types to get a self-consistent subdivision rule.}\label{TorusAddedLines}
\end{center}
\end{figure}

However, this divides the top and bottom squares into two triangles each; each of these are part of a loaded star, so we have to change the replacement rule for a loaded star; however, notice that adding the lines in to the loaded star on the left of Figure \ref{TorusTripleSub} gives us a hexagon divided into six `pie slices'. If we add similar lines bisecting the loaded star (as seen in Figure \ref{TorusAddedLines}), we again get a hexagon divided into six triangles; thus, the subdivision on each triangle in that hexagon is just the identity.

We summarize this in Figure \ref{TorusSubs}. Several stages of subdivision are shown in Figure \ref{CircleTorus} on page \pageref{CircleTorus}. This is a combinatorial subdivision only; the circle packed pictures are not subsets of each other, because this subdivision rule is not conformal. The connection between circle packings and conformality is explained in \cite{French}.

\begin{figure}
\begin{center}
\scalebox{.6}{\includegraphics{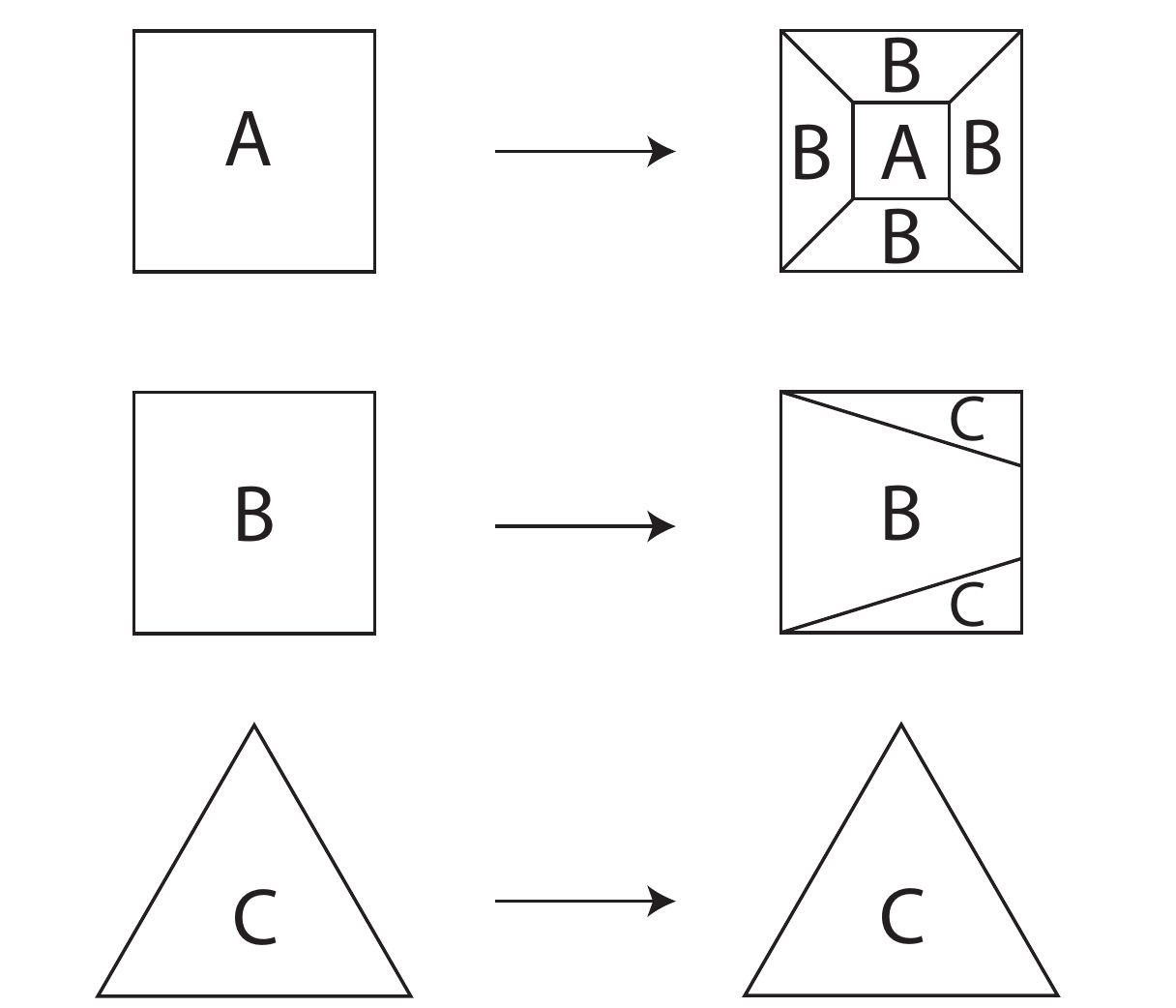}}
\caption{The replacement rule for the 3-dimensional torus.}\label{TorusSubs}
\end{center}
\end{figure}

In general, we can get a subdivision rule if we can add finitely many lines to each replacement type in a way that is self-consistent. However, self-consistency can be difficult to achieve; notice that adding lines in the loaded pair forced us to add lines in the loaded star. In more complicated subdivision rules, adding lines in one tile type can cascade and force us to add more and more lines until we have infinitely many tile types.

There are large classes of polyhedra that do have subdivision rules, however. In the discussion that follows, we will mention the edge cycle lengths of polyhedra many times. Edge cycle length in a hyperbolic manifold can correspond to dihedral angles of polyhedral decompositions. If every dihedral angle is $\frac{2\pi}{n}$, then all edge cycles have length $n$. Polyhedral decompositions with varying dihedral angles do not follow this pattern.

It will be useful to define $n$-cycles, following \cite{Rivin}. An $n$-\textbf{cycle} is a set of $n$ faces $A_1, A_2,...,A_n$ in a polyhedron such that $A_i \cap A_{i+1}$ (subscripts taken mod $n$) is a single edge for each $i$, and all such edges are pairwise disjoint.

The following two theorems give two large classes of polyhedral gluings that can be made into subdivision rules. The first always gives closed 3-manifolds; the second always gives 3-manifolds with boundary.

In the proof that follows, a \textbf{loaded vertex} will be the center vertex of a loaded star.

\begin{thm}\label{SubdivisionTheorem}
Assume a manifold $M$ is formed by gluing together convex, right angled hyperbolic polyhedra $P_1,..,P_n$. Then the replacement rules in Theorem \ref{BigTheorem} can be made into subdivision rules by adding finitely many lines to replacement tile types.
\end{thm}

Note that all such polyhedra satisfy \cite{Rivin}:
\begin{enumerate}
\item each polyhedron $P_i$ is spread out, and
\item each vertex of each polyhedron has valence three,
\item every edge (after gluing) has edge-cycle length four, and
\item there are no three-cycles or four-cycles.
\end{enumerate}

\begin{proof}
We need only show that we can make each replacement of a tile a subdivision of a tile, while retaining finitely many tile types. Because every vertex has valence three, any edges that share a vertex also share an edge, which simplifies the combinatorics significantly (see Figure \ref{VertexProgression}). Throughout the following proof, we will refer to the replacement rule for the three torus as an example. In Figure \ref{Dodecasubs}, we give an example of a more complicated replacement rule, that of the right-angled dodecahedral orbifold mentioned earlier. We will use this replacement rule as an example as well. We proceed by cases.

\begin{figure}
\begin{center}
\scalebox{.6}{\includegraphics{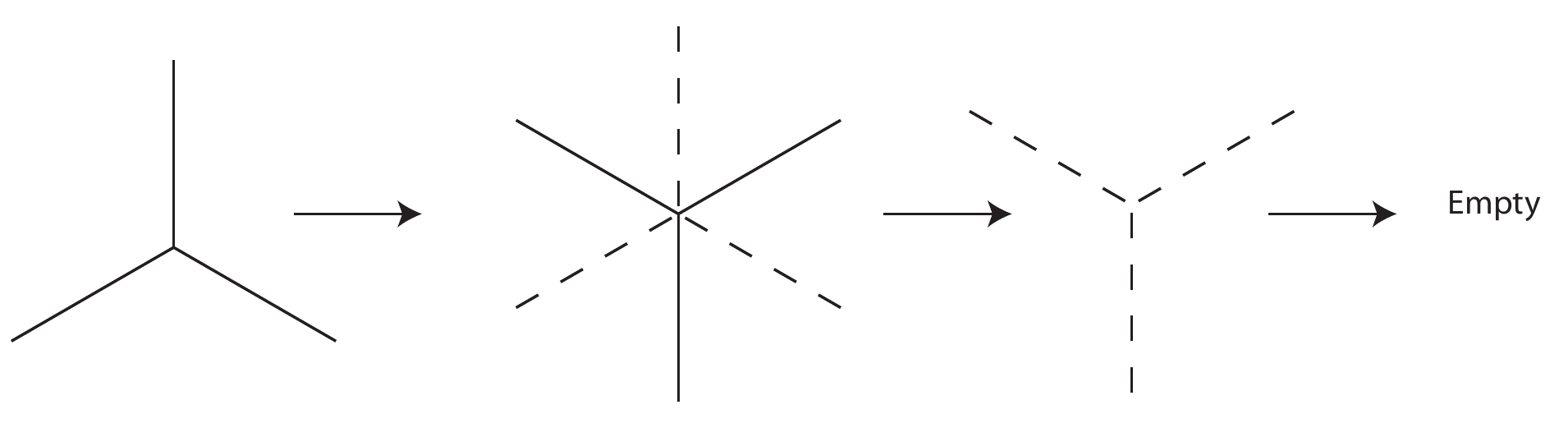}}
\caption[The replacement rule at each vertex.]{Each new vertex starts with three unburdened lines. In the next stage of replacement, three unburdened lines are added, and the old edges become loaded. The vertex is surrounded by three loaded pairs. In the next stage, there are only three edges, and the vertex is now loaded. In the next stage, the vertex and its star are entirely covered up by a new polyhedron.}\label{VertexProgression}
\end{center}
\end{figure}

\begin{figure}
\begin{center}
\scalebox{.5}{\includegraphics{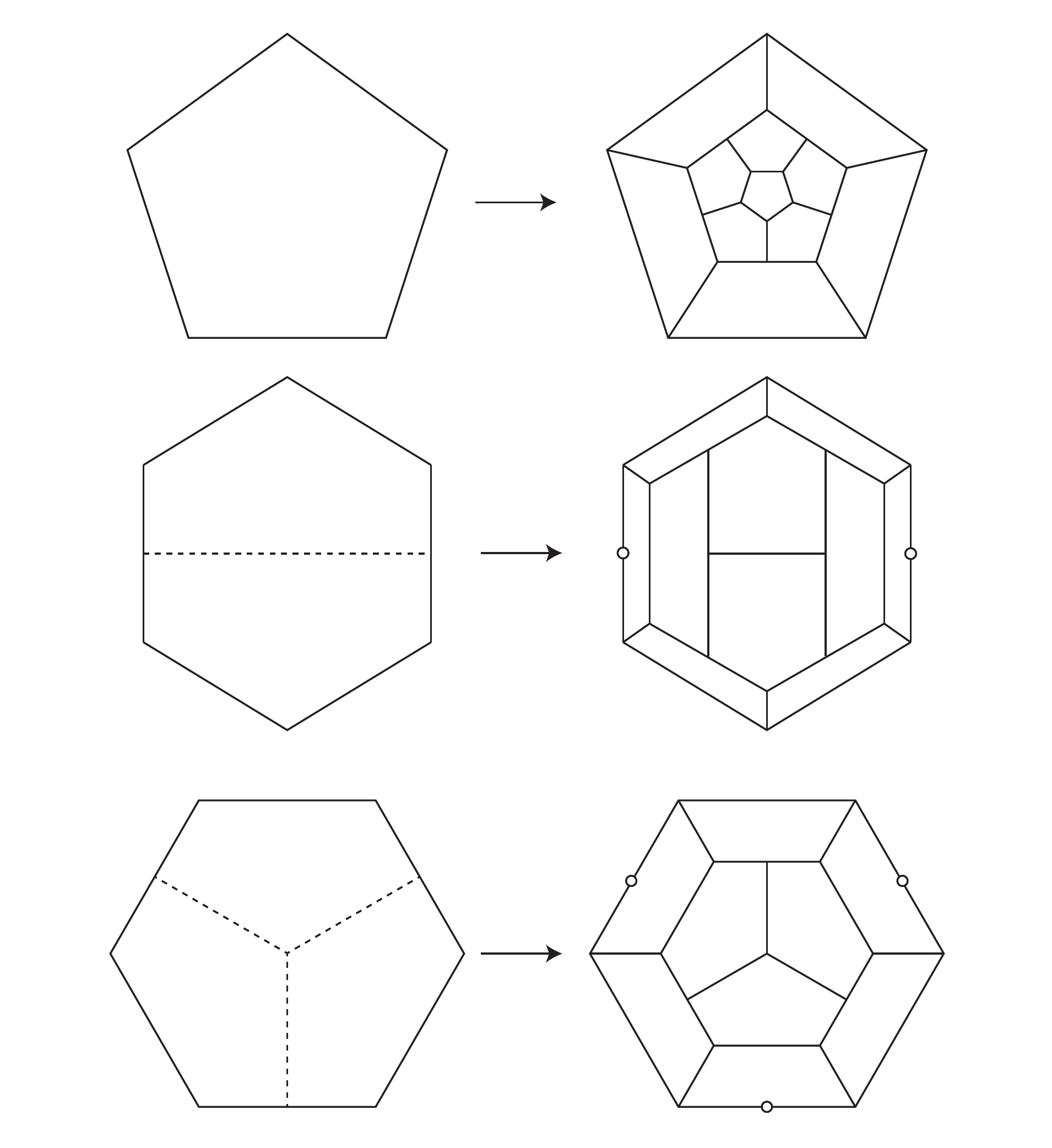}}
\caption{The three replacement types for the right-angled dodecahedral orbifold. They are fans of size one, two, and three.}\label{Dodecasubs}
\end{center}
\end{figure}

\textbf{Case I: Replacing a single unloaded tile in $S(n)$.}

The replacement rule for an unloaded tile is already a subdivision, as no edges disappear. In $S(n+1)$, the only tile types created by this replacement are more unloaded tiles from the interior (which are also case I) and loaded pairs from the boundary (which are case II).

\textbf{Case II: Replacing a loaded pair in $S(n)$.}

Recall that, in creating $S(n+1)$, we glue one polyhedron onto the entire loaded pair, which consists of two faces $A$ and $B$ sharing a common loaded edge $E$ that travels between vertices $v$ and $w$ (see Figure \ref{LabeledLoadedPair}). The cell structure of the unglued faces of the polyhedron (which we call $P^*$) replaces the cell structure of $A\cup B$. To make this replacement a subdivision rule, we find or create a path from $v$ to $w$ in $P^*$ that will take the place of $E$ (as in Figures \ref{TorusAddedLines} and \ref{DodecaPairAddedLines}). Each of these two vertices is contained in a unique tile of $P^* \subset S(n+1)$ (recall Figure \ref{VertexProgression}). We can add an edge from each vertex to an interior vertex of this tile. If the two new edges meet in the same point, we can stop, and identify these two edges with the original, loaded edge $E$.

\begin{figure}
\begin{center}
\scalebox{.6}{\includegraphics{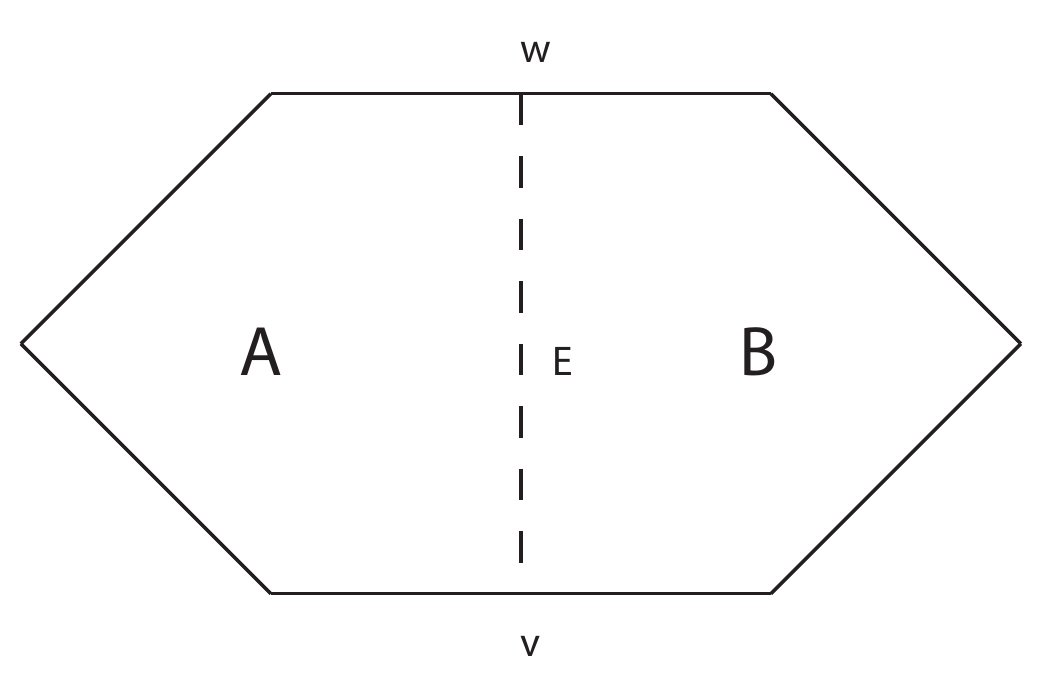}}
\caption{A loaded pair.}\label{LabeledLoadedPair}
\end{center}
\end{figure}

\begin{figure}
\begin{center}
\scalebox{.6}{\includegraphics{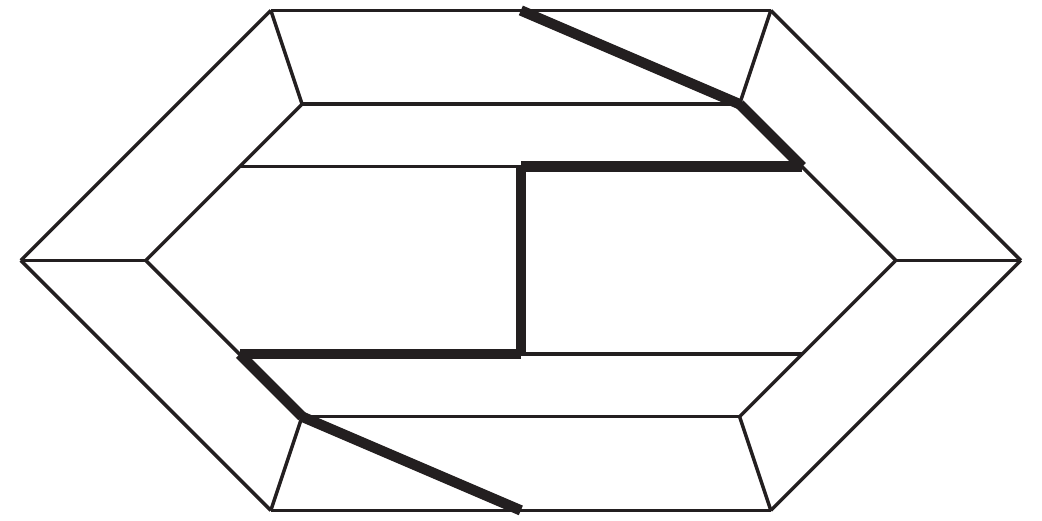}}
\caption[Finding a subdivision rule for the dodecahedral loaded pair.]{We replace the edge E from figure \ref{LabeledLoadedPair} with the path in bold. Note that the initial and terminal segments were not originally in $P^*$.}\label{DodecaPairAddedLines}
\end{center}
\end{figure}

If they meet in different points, then we connect those points by a path consisting entirely of interior edges. If no such path exists, then the set of interior edges is disconnected, which implies that there must be a tile that separates the two points (i.e. the complement of the separating tile is disconnected, with each component of the complement containing interior edges of the separating tile). Because valence is three, that tile must intersect the loaded pair $A\cup B$ in disjoint edges. However, this is impossible by the spread out condition, recall Definition \ref{SpreadOut}.

Thus, we can find a path $\alpha$ in $P^*\subset S(n+1)$ that we can identify with the loaded edge $E \subset A \cup B \subset S(n)$. This makes the replacement rule for a loaded pair into a subdivision rule for each half of the loaded pair, just as in Figure \ref{TorusAddedLines} and Figure \ref{DodecaPairAddedLines}.

Note that, in finding a replacement for $E$ in $P^*$, we only changed the cell structure near $v$ and $w$, by adding one extra line to the tiles containing them; everywhere else, we used the existing cell structure of $P^*\subset S(n+1)$. Each tile containing $v$ or $w$ in $S(n+1)$ is part of a loaded star that will be replaced by a single polyhedron in $S(n+2)$ (recall Figure \ref{VertexProgression}). Since we treated each loaded pair the same way, we see that every loaded star in $S(n+1)$ consists of three tiles around a vertex, each tile sharing a loaded edge with each neighbor and each tile having an added line from the center vertex to an outside, unburdened vertex (as shown in Figure \ref{TorusSubs} and Figure \ref{DodecaTriple}). This is case III. Note that all other tiles in $P^*\subset S(n+1)$ correspond to case I or case II.

\begin{figure}
\begin{center}
\scalebox{.6}{\includegraphics{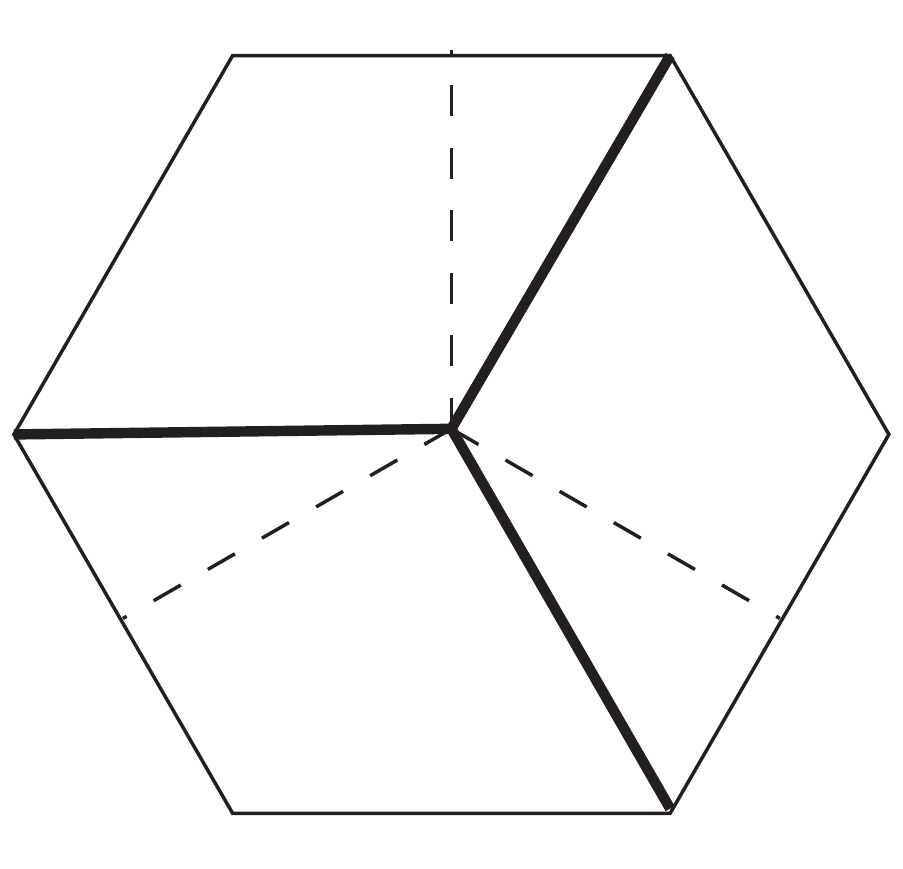}}
\caption{All loaded stars have three loaded lines and three added lines from the previous stage.}\label{DodecaTriple}
\end{center}
\end{figure}

\textbf{Case III: Replacing a loaded star in $S(n)$ with three added radial lines ending at boundary vertices.}

To find a subdivision rule for these loaded stars, we must find or introduce all six of these edges in the new cell structure (which we again call $P^* \subset S(n+1)$). The method we will use is complicated, but works for all right-angled hyperbolic 3-manifolds (it won't work for the three torus example we did earlier). Call the three loaded edges $E_1, E_2$ and $E_3$, and call the added edges $A_1$, $A_2$, and $A_3$. Call the center vertex $V$.

We start by finding lines in $P^*$ that represent the three loaded lines $E_1$, $E_2$, and $E_3$. Each of these lines ends at a vertex $v_1, v_2$ and $v_3$ on the boundary of the star in $S(n)$. As we glue a new polyhedron on to form $S(n+1)$, none of these vertices will have new interior edges coming off of them (they are loaded vertices; see Figure \ref{VertexProgression}). So, we will add edges to $P^*$ to make a path $\alpha$ in $S(n+1)$ that connects $v_1$ and $v_2$. We begin such a path by adding edges $\alpha_1, \alpha_2$ from $v_1$ and $v_2$ to the midpoints of interior edges $e_1, e_2$ (we abuse terminology by referring to an arbitrary interior point of an edge as the `midpoint' of an edge, to avoid the overuse of the word interior). We then connect the endpoints of $\alpha_1, \alpha_2$ by a series of added edges $\alpha_3,...,\alpha_n$, each of which goes from the midpoint of one interior edge $e_i$ in $P^*$ to the midpoint of another, interior edge $e_j$ in $P^*$. We define the path $\alpha$ to be $\cup \alpha_i$.

Note that we can assume that the $e_i$ are disjoint, because if they are not disjoint, we can shorten the path $\alpha$ (see Figure \ref{ShorteningAlpha}). Similarly, we can assume that no tile contains more than one $\alpha_i$. Finally, we can assume that $\alpha$ does not intersect any boundary tiles (including their interior edges) except for the two tiles containing $\alpha_1$ and $\alpha_2$. We can do this because we claim that the complement of the boundary tiles in $P^*$ is connected. To prove this claim, note that the only way the complement could be disconnected would be for two boundary tiles to intersect in an interior edge. However, this would imply that our polyhedron contains a four cycle (four faces which intersect cyclically in four disjoint edges), which is impossible by Andreev's theorem \cite{Andreev}. See Figure \ref{NoFourCycles}. This is where the proof fails for non-hyperbolic manifolds such as the 3-torus.

\begin{figure}
\begin{center}
\scalebox{.6}{\includegraphics{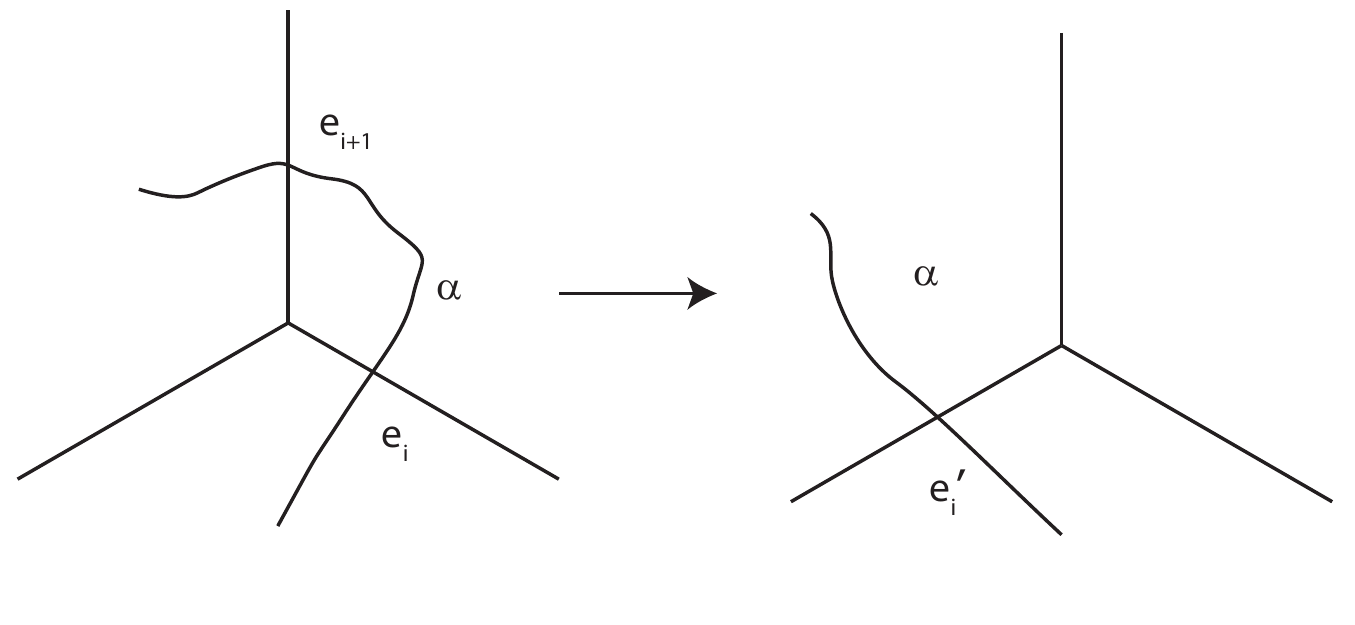}}
\caption{A path that goes through non-disjoint edges can be shortened.}\label{ShorteningAlpha}
\end{center}
\end{figure}

\begin{figure}
\begin{center}
\scalebox{.6}{\includegraphics{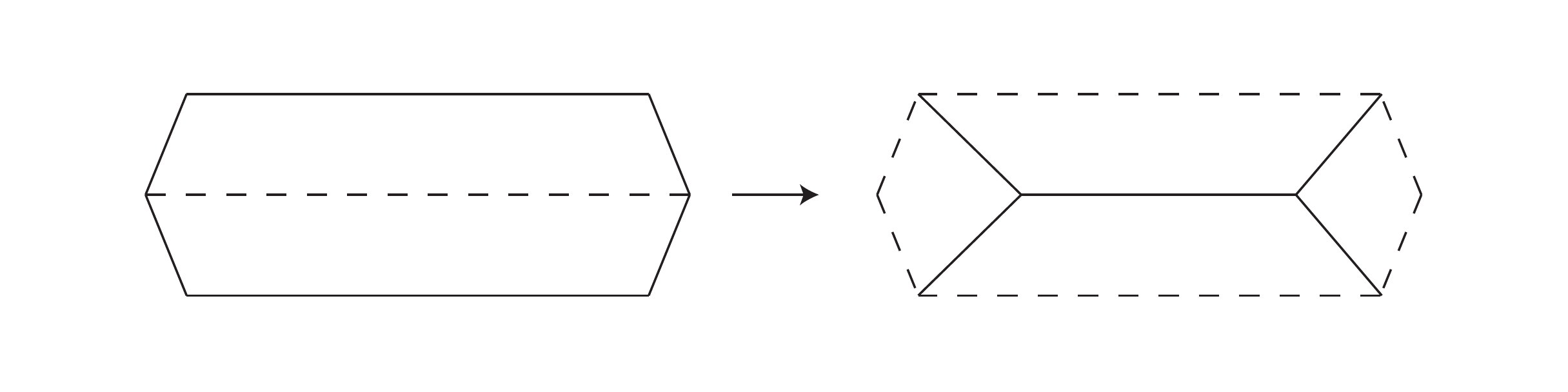}}
\caption{If the complement of the boundary tiles is disconnected, there must be a four cycle.}\label{NoFourCycles}
\end{center}
\end{figure}

Thus, as said above, we can connect $v_1$ to $v_2$ by a path $\alpha$ which is formed from new edges, each joining two disjoint edges in a tile, except for the initial and terminal segments, with no tile containing two new added edges. Now, we form another path $\beta$ with the same properties, but now connecting $v_3$ to $\alpha$. Again, we start the path by adding an edge from $v_3$ to an interior edge, and then continue the path by adding line segments containing disjoint edges until we reach an edge $e$ of an interior tile $T$ containing $\alpha$. The terminal segment connects $e$ to the midpoint of the segment of $\alpha$ contained in $T$. See Figure \ref{AlphaAndBetaPaths}.

\begin{figure}
\begin{center}
\scalebox{.6}{\includegraphics{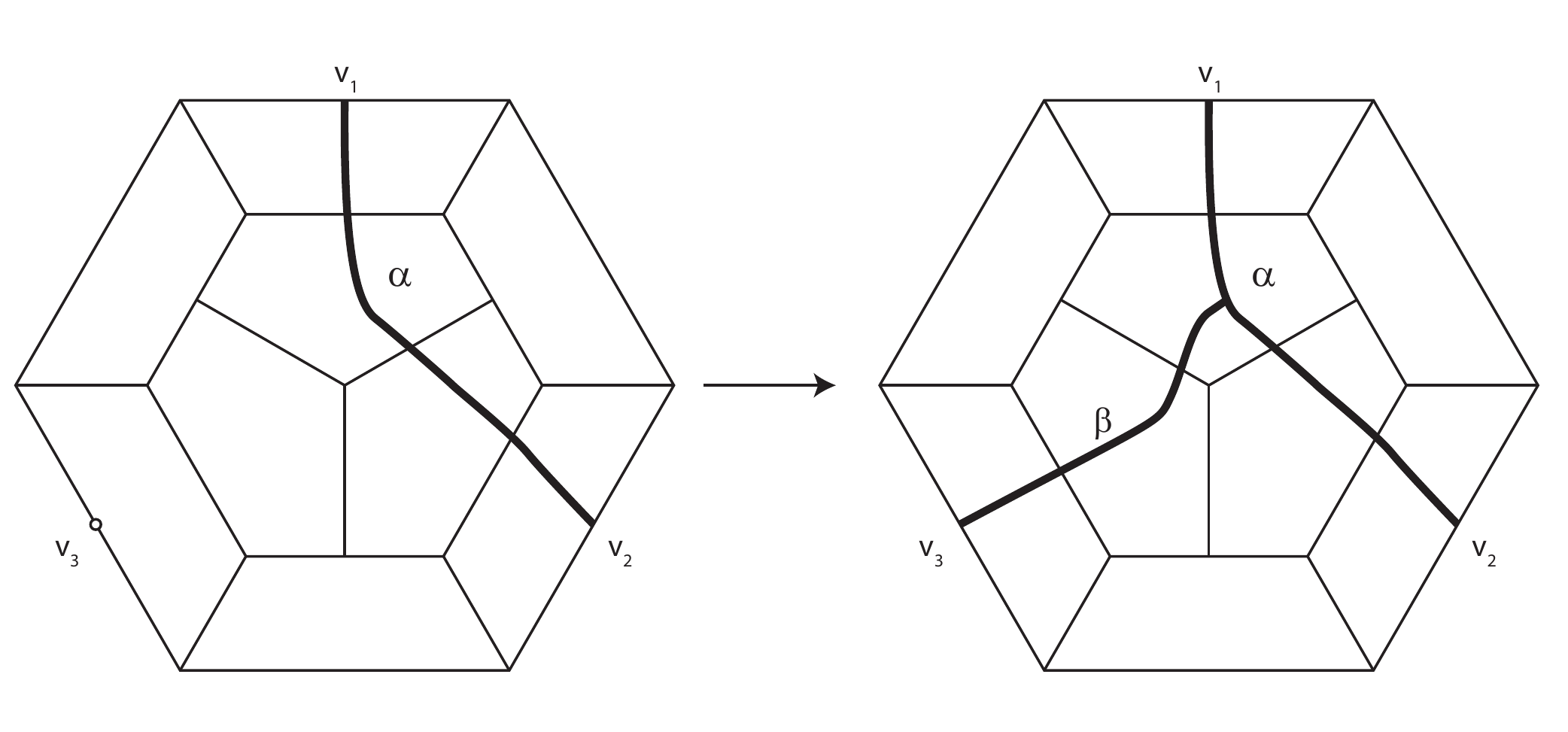}}
\caption{We find paths $\alpha$ and $\beta$ to replace the loaded edges of the loaded star.}\label{AlphaAndBetaPaths}
\end{center}
\end{figure}

Thus, we can add $\alpha\cup\beta$ to $P^*$ and identify $\alpha\cup\beta \subseteq S(n+1)$ with $E_1 \cup E_2 \cup E_3 \subset S(n)$. We identify $\alpha\cap\beta$ with the vertex $V=E_1\cap E_2 \cap E_3$. We have now replaced all loaded edges.

Now, we replace $A_1, A_2,$ and $A_3$, which, as we recall, represent the lines added in Case II. They intersect the boundary of our star in $S(n)$ in vertices $w_1, w_2,$ and $w_3$, respectively. Each vertex $w_i$ lies in a distinct component $C_i$ of $P^* \setminus \{\alpha \cup \beta \cup T^o\}$, where $T^o$ is the interior of the tile containing $\alpha \cap \beta$. We claim that the set of interior edges of each $C_i$ is connected. To see this, note that any gap in the interior edges (i.e. having more than one connected component of edges) would be caused by a single tile of $P^*$ being crossed by $\alpha \cup \beta$ in two or more disjoint segments. However, by construction, this never occurs, except in $T$. Thus, the set of interior edges in each $C_i$ is connected.

Also, each component $C_i$ must contain a vertex of $T_i$, since $\beta$ does not enter $T$ by the same edge as $\alpha$. Thus, we can find a path $\gamma_i$ from each $w_i$ to a vertex $t_i$ of $T$. Extend each $\gamma_i$ by an edge from $t_i$ to $V$ in $T$. See Figure \ref{GammaPaths}. We now identify each $\gamma_i$ with the appropriate $A_i$. Thus, by adding $\alpha, \beta, \gamma_1, \gamma_2$, and $\gamma_3$ to $P^*$, we can embed the cell structure of the loaded star in $S(n)$ into the cell structure of the replacement in $S(n+1)$.

\begin{figure}
\begin{center}
\scalebox{.6}{\includegraphics{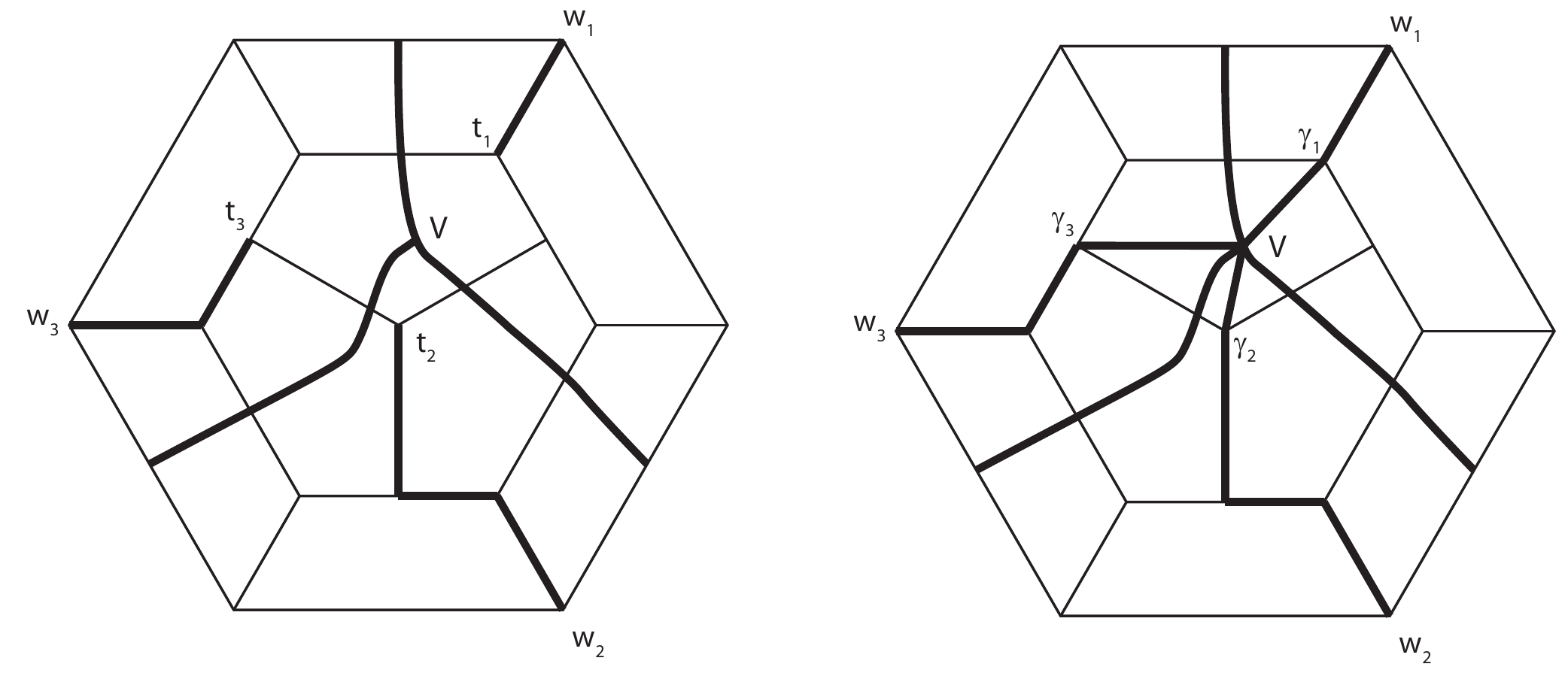}}
\caption{We find paths $\gamma_1, \gamma_2,$ and $\gamma_3$ to replace the three added lines of the loaded star.}\label{GammaPaths}
\end{center}
\end{figure}

However, this creates three new tile types: unloaded tiles with a line added connecting midpoints of disjoint edges, loaded stars with three added radial lines that may now end at midpoints of edges instead of vertices, and an unloaded tile with six added edges (three intersecting the boundary in vertices, three intersecting in midpoints of edges, the two types alternating around the boundary). These are cases $IV$, $V$, and $VI$, respectively.

\textbf{Case IV: Replacing a single unloaded tile with an added line connecting midpoints of edges in $S(n)$.}

Call the added line $A\subseteq S(n)$. Let $e_1$ and $e_2$ be the two boundary edges connected by $A$. Then in $S(n+1)$, we can add an edge going from the midpoint of each $e_i$ to an interior vertex. We can then connect these two new edges by a path of interior edges. This creates a new case, Case VII, which consists of a loaded pair with one or two added lines going from a boundary vertex to the midpoint of the loaded edge of the loaded pair. See Figure \ref{MidpointToMidpoint}.

\begin{figure}
\begin{center}
\scalebox{.6}{\includegraphics{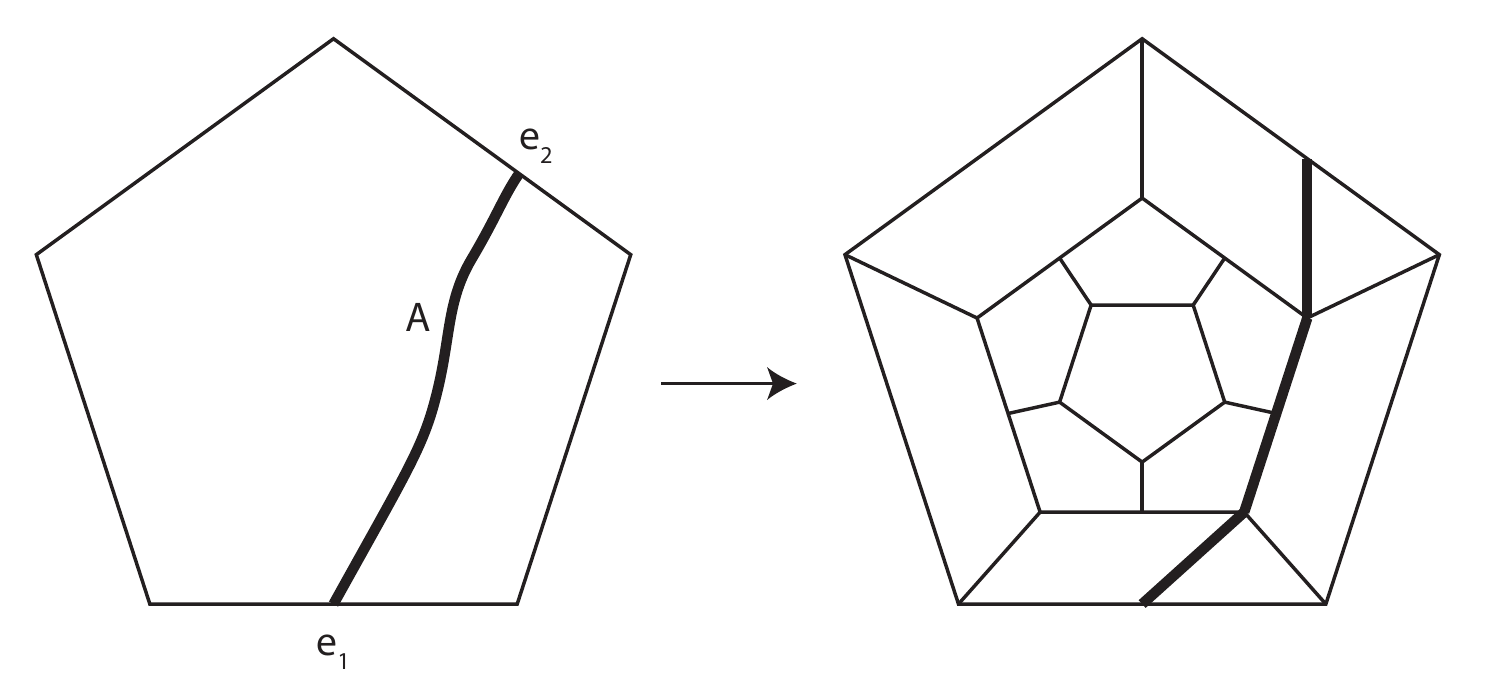}}
\caption{Replacing a single unloaded tile with an added line from midpoint to midpoint.}\label{MidpointToMidpoint}
\end{center}
\end{figure}

\textbf{Case V: Replacing a loaded star in $S(n)$ with three added radial lines, possibly terminating at midpoints of edges.}

Loaded stars of this type will be replaced almost exactly as other loaded stars were earlier; in particular, there are still three loaded edges and three added edges. We still find paths $\alpha$ and $\beta$ to replace all added edges. However, these paths cannot consist entirely of interior edges of $P^*$, as the midpoints of boundary edges are not connected to interior edges. Thus, we begin $\alpha$ and $\beta$ by adding lines from each of these boundary midpoints to an interior vertex. We then extend $\alpha$ and $\beta$ as needed by interior edges, and the proof goes through in exactly the same way. This creates no new tile types. See Figure \ref{TripleWithMidpoints}.

\begin{figure}
\begin{center}
\scalebox{.6}{\includegraphics{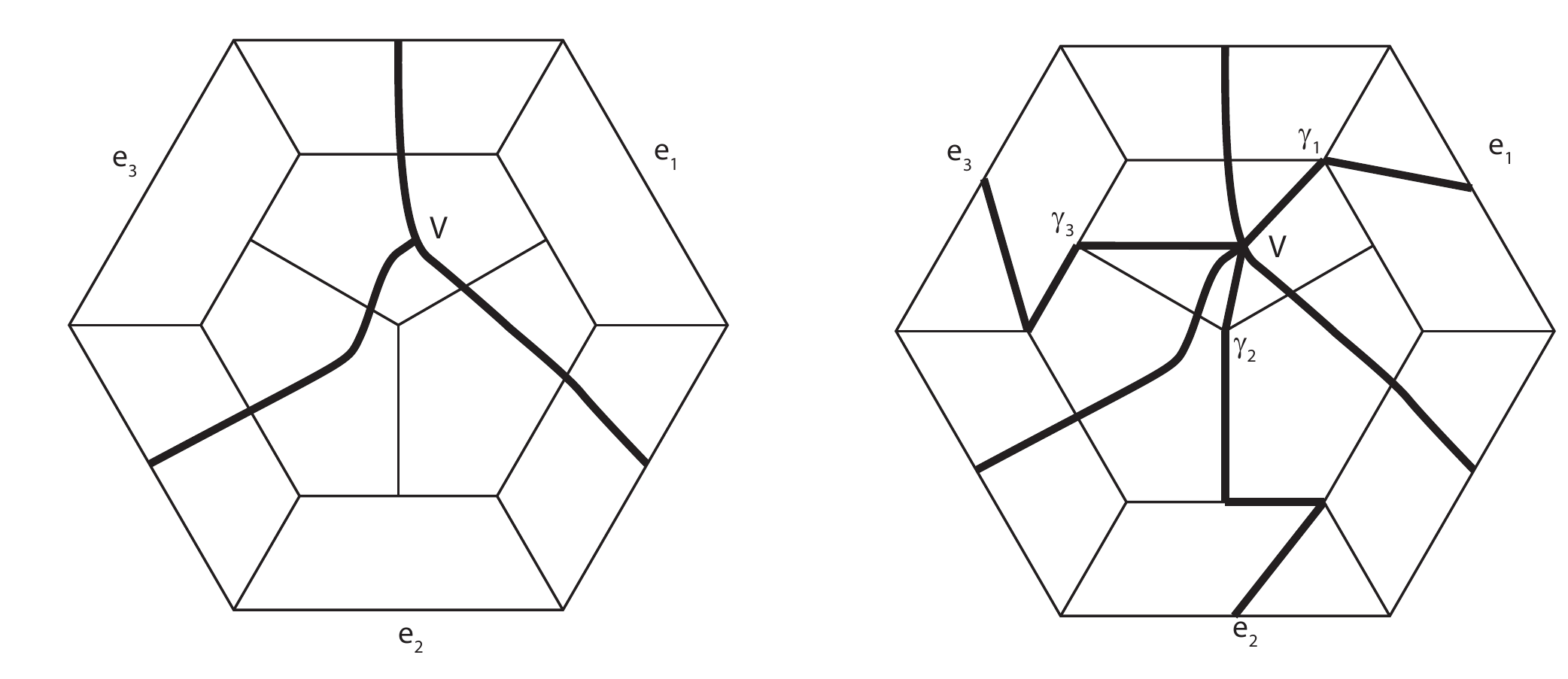}}
\caption{Replacing a loaded star with added edges going to midpoints instead of boundary vertices.}\label{TripleWithMidpoints}
\end{center}
\end{figure}

\textbf{Case VI: Replacing a single unloaded tile with six added radial lines (every other one going to a boundary vertex, the rest to midpoints of edges) in $S(n)$.}

In this case, we are basically doing Case III for a single tile instead of a loaded star. The complement of the boundary in $(S(n+1)$ is still connected, because we would have three-cycle if it were not. So let $A_1, A_2$ and $A_3$ be the added lines ending at vertices $v_1, v_2$ and $v_3$, and let $E_1$, $E_2$ and $E_3$ be the added lines ending at edges $e_1$, $e_2$, and $e_3$.

As a modification of Case III, we can connect the midpoints of $e_1$ and $e_2$ with a path $\alpha$ of added edges connecting disjoint edges of non-boundary tiles (except for the initial and terminal segment). We can similarly construct $\beta$ connecting the midpoint of $e_3$ and $\alpha$. After this, we can construct $\gamma_1, \gamma_2$ and $\gamma_3$ in exactly the same way as Case III and get a subdivision. This creates more tiles of Case IV, VI, and VII.

\begin{figure}
\begin{center}
\scalebox{.6}{\includegraphics{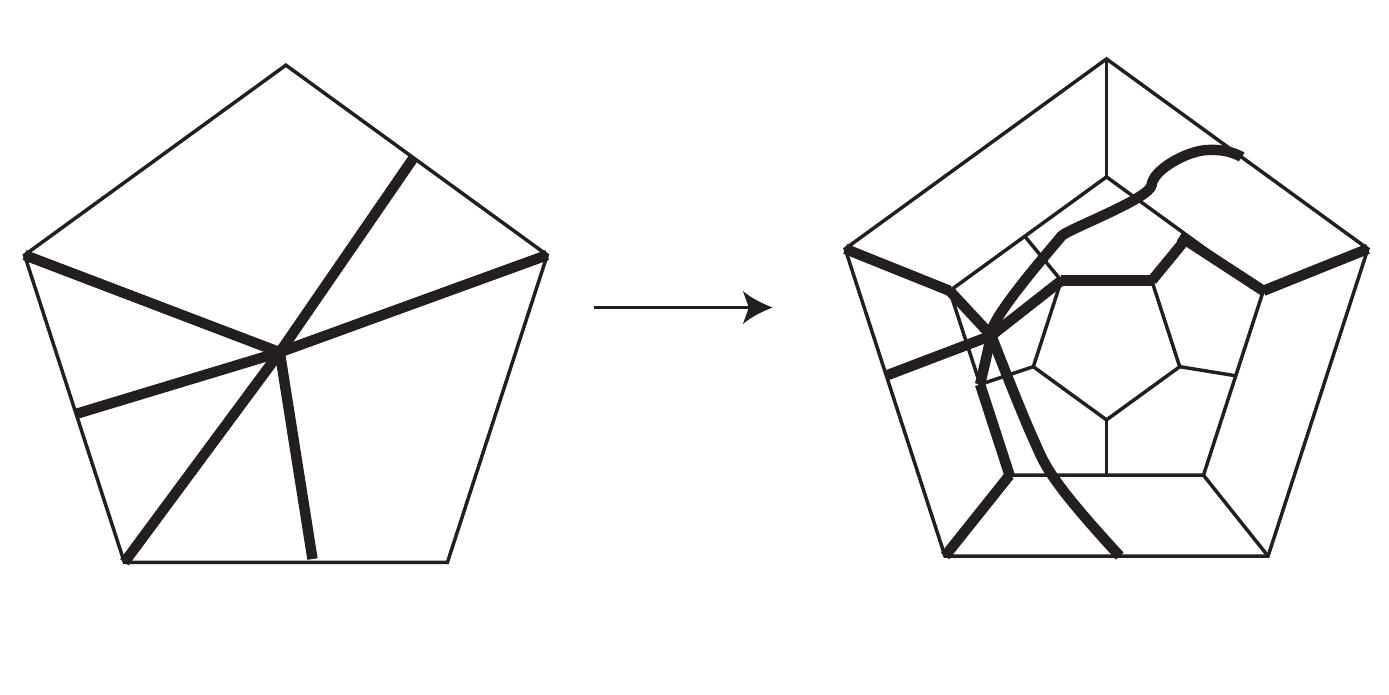}}
\caption{Replacing a single unloaded tile with six added lines.}\label{SingleWithSixAddedLines}
\end{center}
\end{figure}

\textbf{Case VII: Replacing a loaded pair with two edges coming from the boundary to the midpoint of the loaded edge of the loaded pair. $S(n)$.}

First, assume that the two added edges $A_1$, $A_2$ meeting in the midpoint of the loaded edge $E$ have their other endpoints at two boundary vertices $w_1$ and $w_2$. Call the loaded edge $E$.

Instead of replacing $E$ as in case II, we replace it by a path $\alpha$ as in Case III that travels from the endpoints of $E$ and only goes through midpoints of interior edges. This creates more tiles of Case $V$ and Case $IV$. Then connect $w_1$ and $w_2$ by a path $\gamma$ of interior edges. This creates no new cases, and $\alpha \cup \gamma$ replaces $E$ and $A_1$, $A_2$.

The case when $A_1$ and/or $A_2$ have their other endpoints at midpoints of boundary edges is similar; however, we begin and end $\gamma$ with segments attaching the midpoints of the boundary edge to an arbitrary interior vertex; the rest of gamma consists of interior edges. This creates no new cases.

\begin{figure}
\begin{center}
\scalebox{.6}{\includegraphics{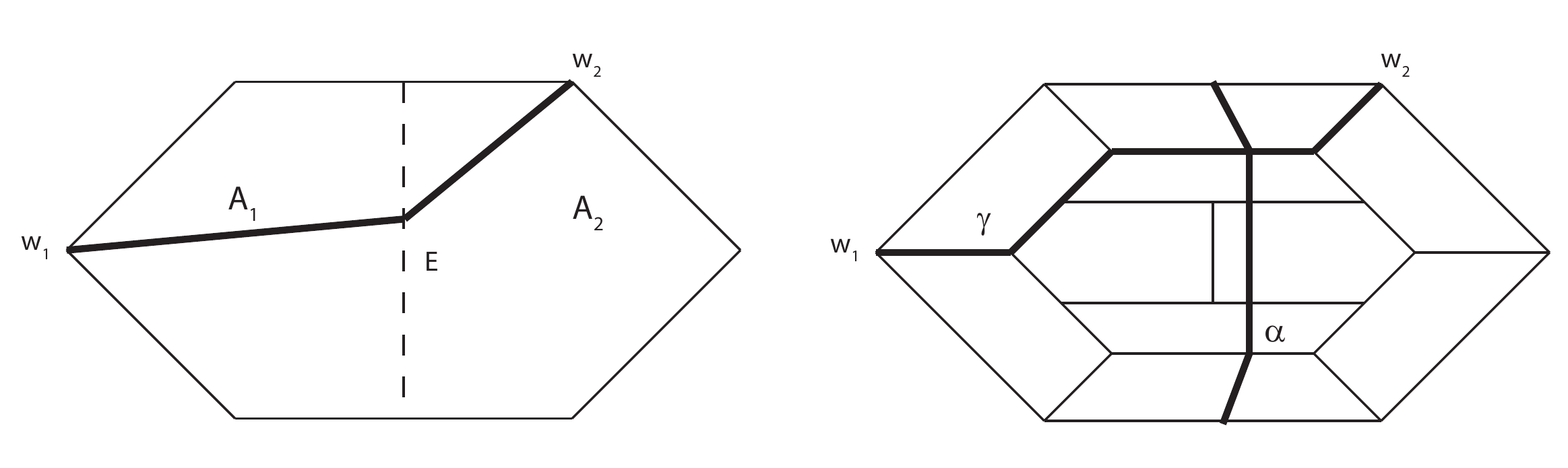}}
\caption{Replacing a loaded pair with added lines starting at the boundary and meeting in the midpoint of the loaded edge.}\label{LoadedWithMidpoints}
\end{center}
\end{figure}

\textbf{End of cases}

So note that each of our seven cases of tile types generates more tiles, but only within those seven cases. However, each tile type is a slight modification of a loaded fan, of which there are finitely many. There are also only finitely many ways to modify each of these finitely many fans (by adding lines). Thus, there are finitely many tile types.

Thus, by adding finitely many lines to the replacement tile types, we can turn the replacement rule of Theorem \ref{BigTheorem} for such a manifold into a subdivision rule with finitely many tile types.

\end{proof}

This theorem applies only to closed 3-manifolds (since the link of each vertex is a sphere). If, in some spread out polyhedral gluing with even edge cycle lengths, the valence of each polyhedron is uniformly greater than three or the edge cycle lengths are all six or larger, we can also find a subdivision rule. However, these manifolds are never closed manifolds, as they have torus or hyperbolic surface boundary. The link of each vertex in the manifold is composed of polygons with a number of edges equal to the valence of the corresponding vertex on the polyhedron; each vertex of these polygons corresponds to a an edge in the manifold, and the valence of the vertex is the edge cycle length of edge. Under these circumstances, (still assuming large valence or large edge cycle length), Chapter 9 of \cite{myself} ensures that only loaded pairs and unloaded tiles occur in the replacement rules for the manifold. In particular, we have the following:

\begin{thm}
Let $M$ be a 3-manifold created by gluing together polyhedra $P_1,...,P_n$ such that each edge cycle has even length $>2$. If each $P_i$ is spread out and:
\begin{enumerate}
\item all vertices have valence greater than 3, or
\item all edges have cycle length greater than 4, then
\end{enumerate}
the replacement rules in Theorem \ref{BigTheorem} can be made into subdivision rules by adding finitely many lines to replacement tile types.
\end{thm}\label{BoundaryTheorem}

\begin{proof}
Recall that the only tile types in the replacement rule from Theorem \ref{BigTheorem} are loaded fans.

The link of every vertex in our manifold is a Euclidean or hyperbolic surface, by an Euler characteristic argument. As we construct the universal cover for our manifold, we also construct, at each vertex, the universal cover for the link. Each vertex of the link corresponds to an edge of the manifold, and each edge of the link corresponds to a face of the manifold. These surfaces have replacement rules as well, with `loaded vertices' corresponding to loaded edges in the manifold. Such replacement rules were studied in Chapter 9 of \cite{myself}. According to the proof of Theorem 12 of that chapter, all surfaces of sufficient size (including those coming from manifolds satisfying the conditions of this theorem) have a replacement rule in which only isolated loaded vertices occur (i.e. no two loaded vertices are ever adjacent). Since loaded vertices in the surface correspond to loaded edges in the manifold, this shows that every loaded edge must be bordered by unloaded edges to either side. This shows that only loaded pairs and unloaded tiles can occur.

Thus, if we can replace the loaded edge in every loaded pair, we will be done. However, the proof of Theorem 12 of \cite{myself} also shows that every loaded vertex in the universal cover of the surface is replaced by one or more unloaded vertices. This implies that, in the universal cover of the manifold, the two vertices of each loaded edge in $S(n)$ will both have new interior edges coming from them in $S(n+1)$. We can construct a path in $S(n+1)$ between these two vertices consisting entirely of interior edges, because the only obstacle to such a path would be a tile that separates the loaded pair, which cannot occur by the spread out condition (Definition \ref{SpreadOut}). No new tile types are created by finding such a path, and so we have a subdivision rule.
\end{proof}

As you can see, it is actually easier to find a subdivision rule for 3-manifolds with boundary. Subdivision rules for alternating links (which have valence four, edge cycle length four polyhedral gluings) were found previously in \cite{linksubs}. Also, note that the full `spread out' condition is not needed; we need only check that no tile has disconnected intersection with single tiles or loaded pairs, or intersects in more than two edges.

The tilings in Figures \ref{CircleBorro} to \ref{CircleCubeBig}, and Figures \ref{CircleDod} and \ref{CircleDodBig} were all created by the methods detailed in Theorems \ref{SubdivisionTheorem} and \ref{BoundaryTheorem}. The pictures were created using Ken Stephenson's Circlepack \cite{Circlepak} and software by Bill Floyd \cite{floyd}.

Figures \ref{CircleBorro} and \ref{CircleBorroBig} represent the Borromean rings, a manifold created from two octahedra, each with edge cycle lengths of 4. It is a finite-volume hyperbolic manifold.  Figure \ref{CircleTet} is another finite-volume hyperbolic manifold: the figure-eight knot.

Figures \ref{CircleTorus}-\ref{CircleCubeBig} are all created from cubes. The first figure is the 3-dimensional torus, with edge cycle lengths of 4. It is a Euclidean manifold. The last two figures are a cube with edge cycle lengths of six. It has torus boundary.

Finally, Figures \ref{CircleDod} and \ref{CircleDodBig} show the tilings from a subdivision rule created from the dodecahedral orbifold first studied by Cannon, Floyd and Parry by an adaptation of our method. This is a dodecahedron with edge cycle length four. It is interesting to compare this to their original subdivision rule \cite{subdivision}. The tiles are shown in Figure \ref{Dodecasubs2}.

\begin{figure}
\begin{center}
\scalebox{.5}{\includegraphics{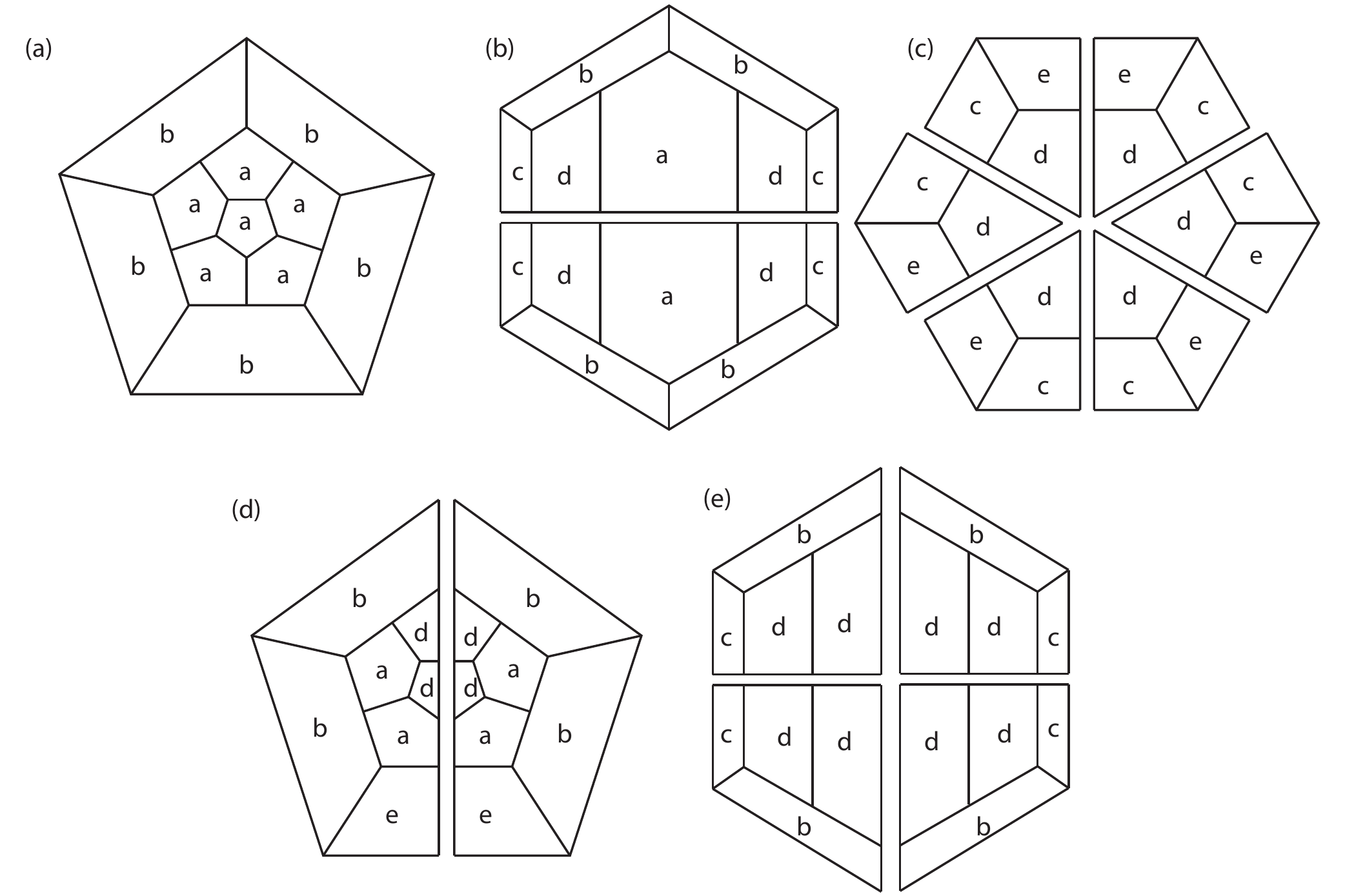}}
\caption[A subdivision rule for a dodecahedral orbifold.]{A subdivision rule obtained from the replacement rule in Figure \ref{Dodecasubs}. We use a slightly different method than the one outlined in Theorem \ref{SubdivisionTheorem} to reduce the number of tile types.}\label{Dodecasubs2}
\end{center}
\end{figure}

Closed hyperbolic manifolds with even edge cycle length are easy to create: start with a set of glued-together polyhedra with large valence at each vertex and even edge cycle length greater than 2. These manifolds will usually have hyperbolic surfaces at the boundary. If we expand each vertex into a face, we get a subdivision rule with valence three, and certain faces that never subdivide (these represent the boundary at infinity). If we double the manifold over its boundary, all edges of the blown-up vertices will have edge-cycle length four, and the other edges retain their original edge cycle length. This gives a closed hyperbolic manifold with even edge-cycle length, as desired.

As a final note, odd edge-cycle length polyhedra are more difficult to work with, especially length 3, but many of the same principles apply. One open problem is to find a concise set of conditions on odd edge-cycle length polyhedral gluings that ensures a subdivision rule exists. This is interesting, because all 3-manifolds have a decomposition into valence 3, edge-cycle length 3 polyhedra (by virtue of the Heegard splitting). It is unknown how many hyperbolic 3-manifolds have a decomposition into even edge-cycle length polyhedra.

\begin{figure}
\begin{center}
\scalebox{0.8}{\includegraphics{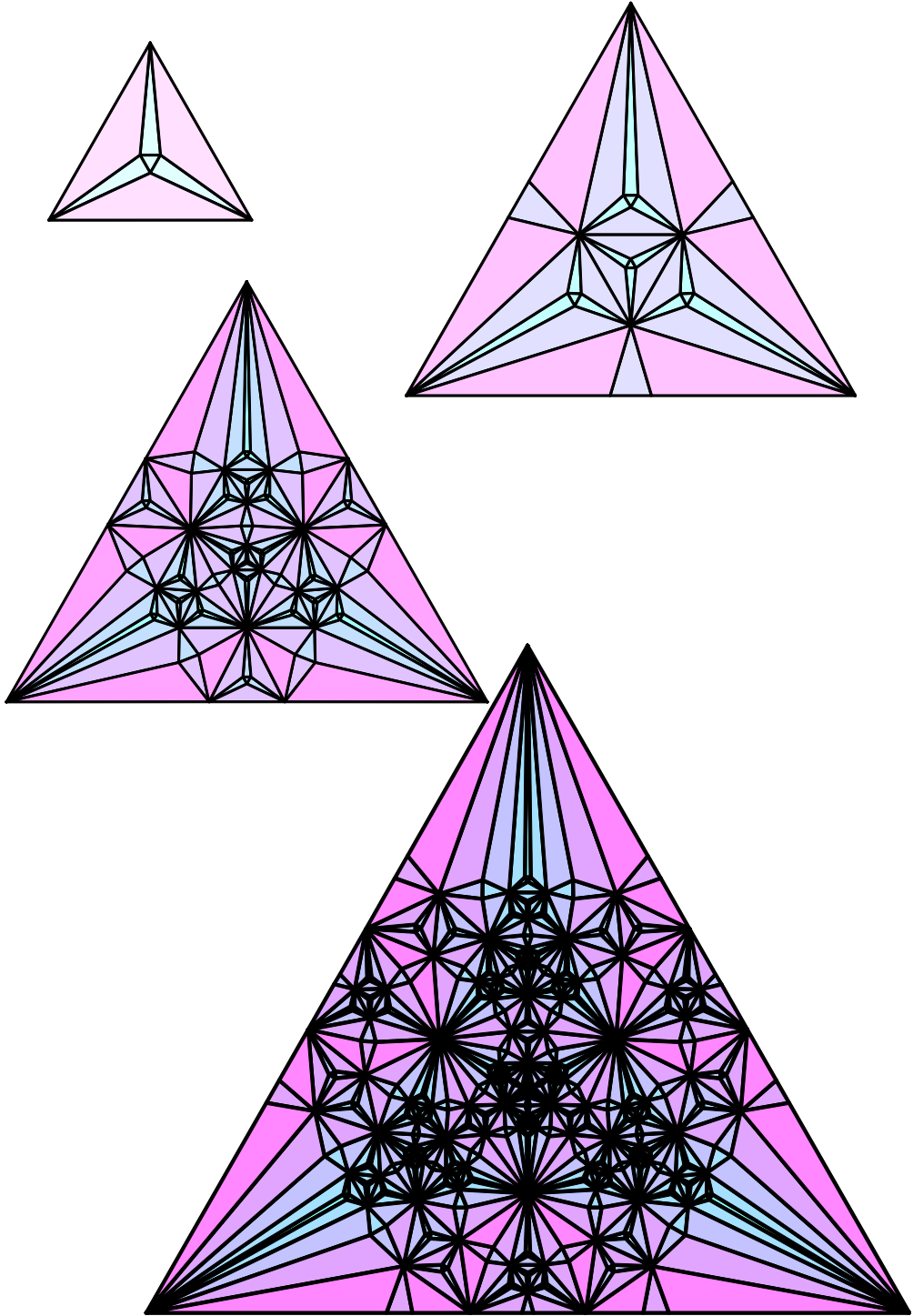}}\caption{The first few subdivisions of the Borromean rings complement, which decomposes into two octahedra with edge cycle length 4.}
\label{CircleBorro}
\end{center}
\end{figure}

\begin{figure}
\begin{center}
\scalebox{0.8}{\includegraphics{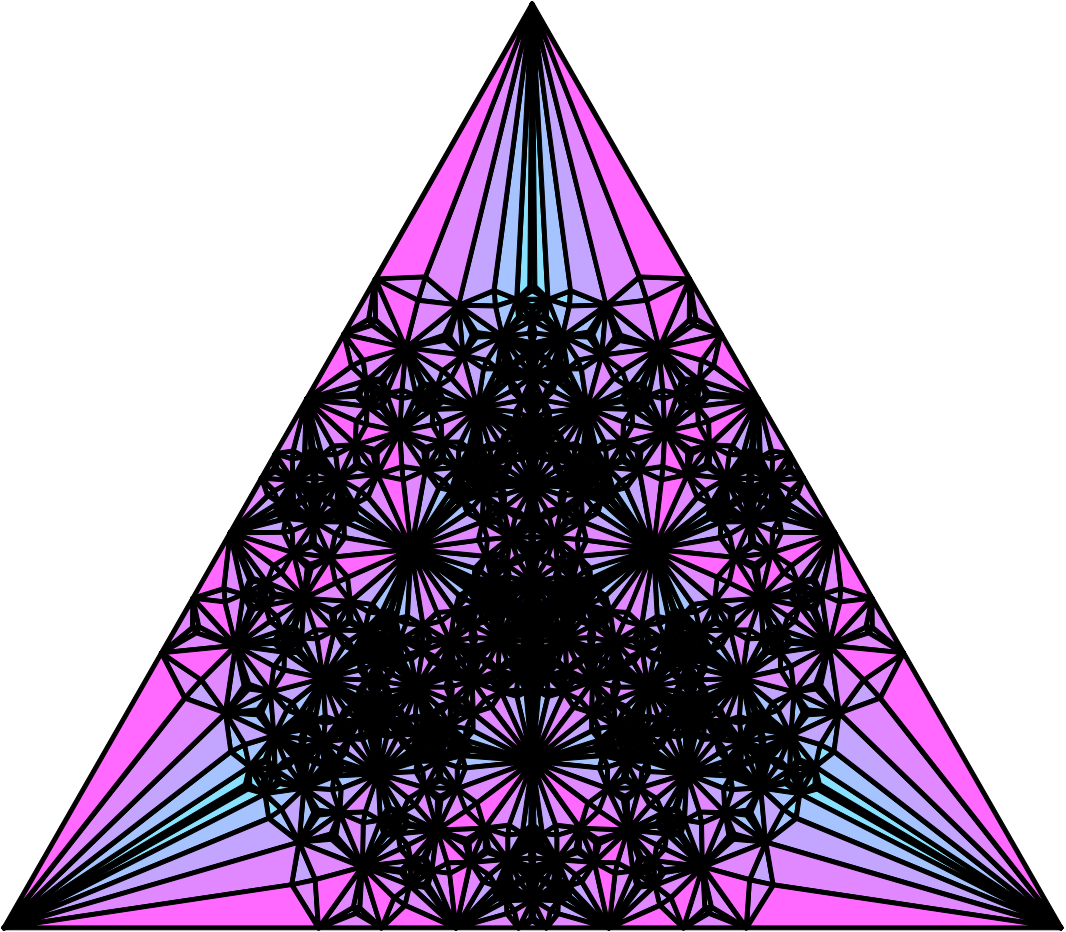}}\caption{One more subdivision of the Borromean rings complement.}
\label{CircleBorroBig}
\end{center}
\end{figure}

\begin{figure}
\begin{center}
\scalebox{.8}{\includegraphics{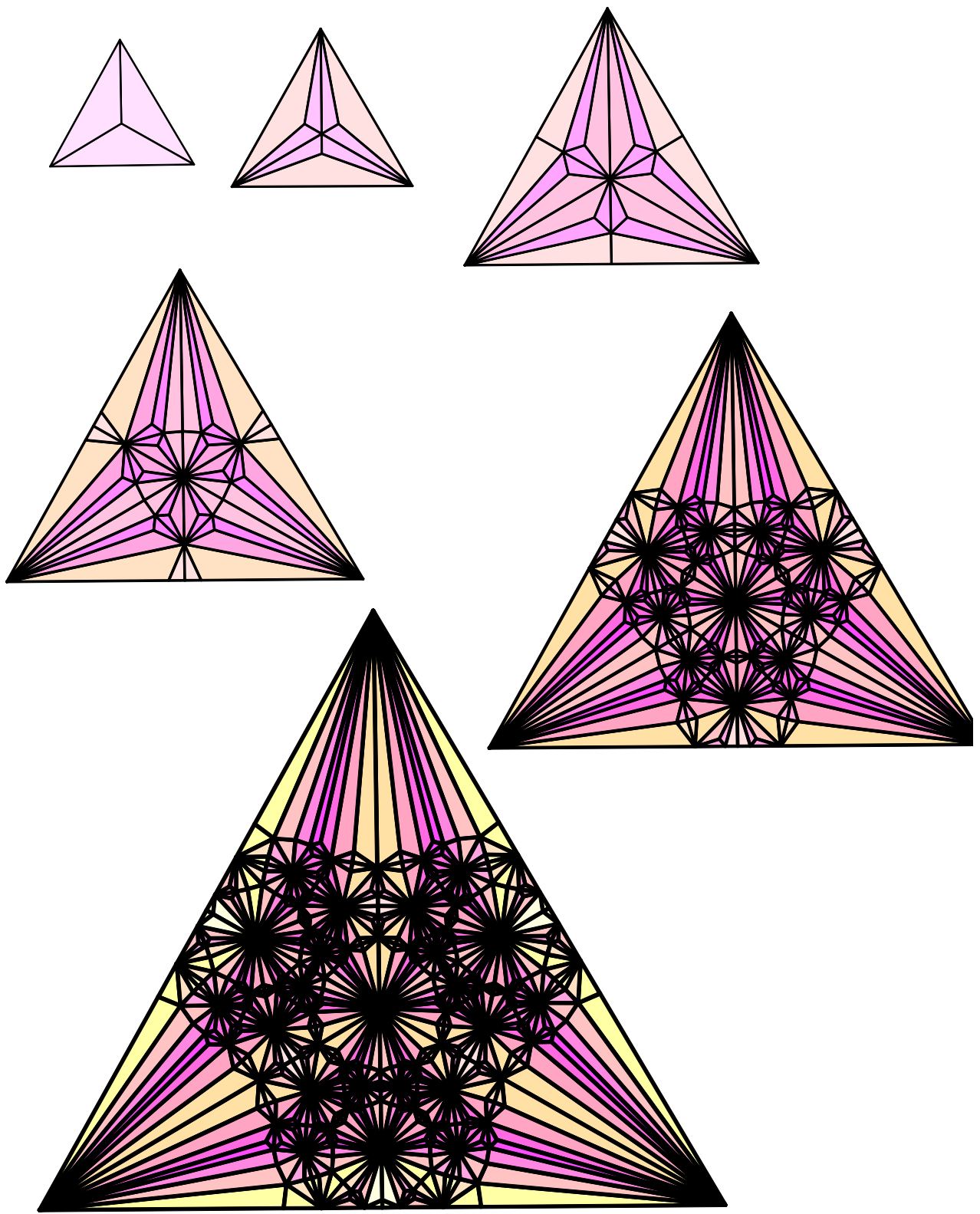}}\caption{The first few subdivisions of the figure eight knot, decomposed into two regular ideal tetrahedra.}
\label{CircleTet}
\end{center}
\end{figure}

\begin{figure}
\begin{center}
\scalebox{.8}{\includegraphics{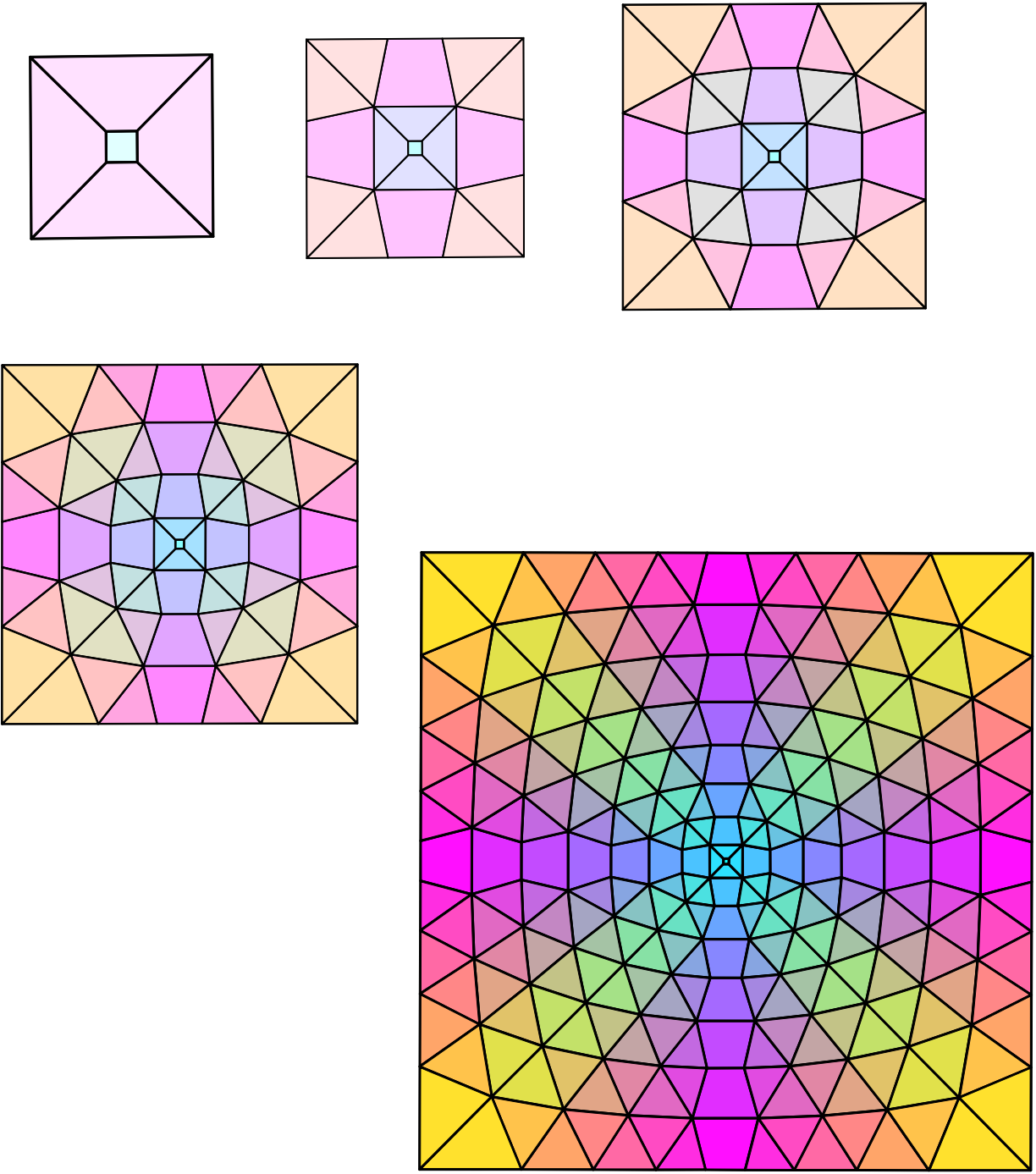}}\caption{Several subdivisions of the 3-dimensional torus, with fundamental domain a Euclidean cube.}
\label{CircleTorus}
\end{center}
\end{figure}

\begin{figure}
\begin{center}
\scalebox{.8}{\includegraphics{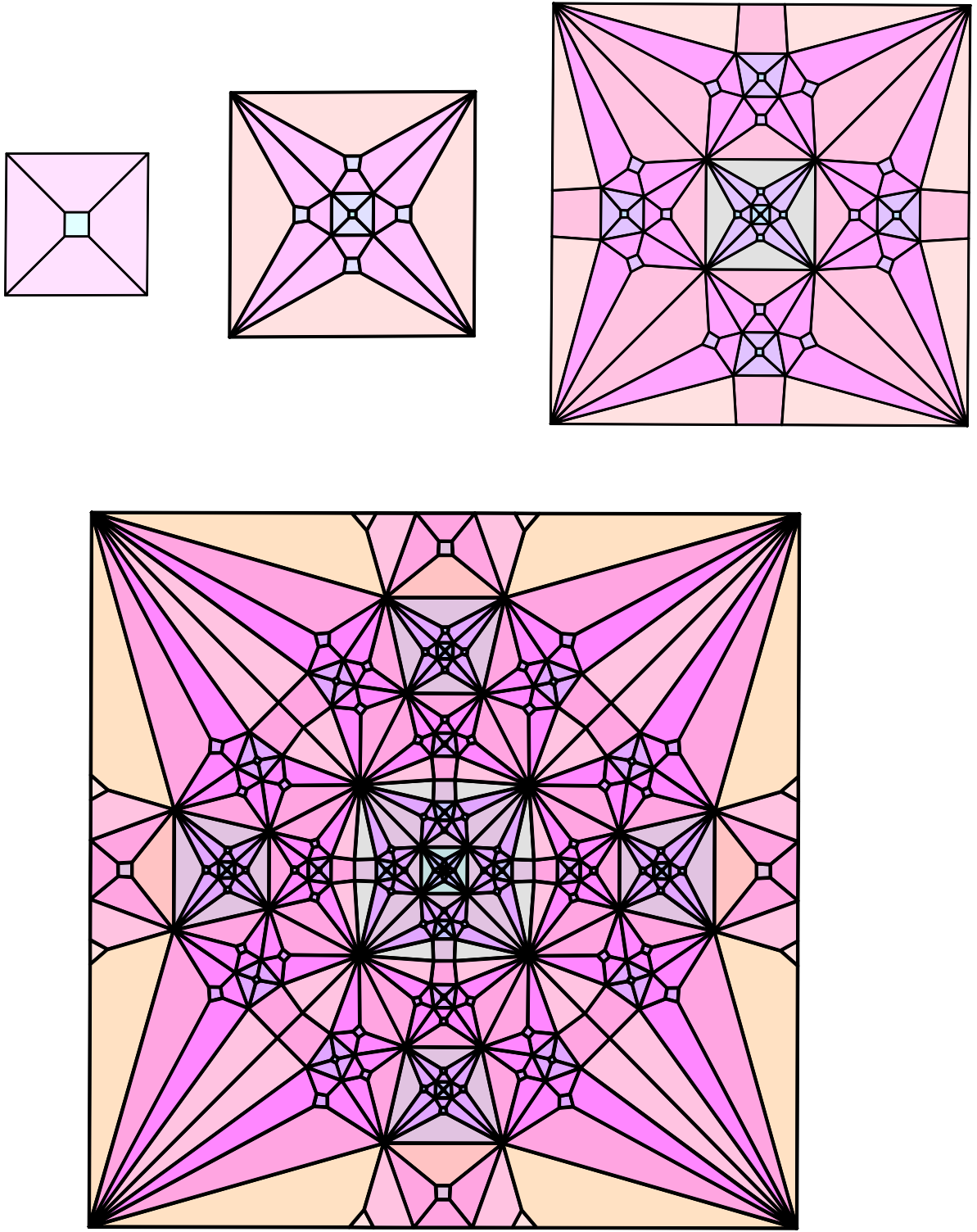}}\caption{Several subdivisions for a cube with edge-cycle length 6.}
\label{CircleCube}
\end{center}
\end{figure}

\begin{figure}
\begin{center}
\scalebox{.8}{\includegraphics{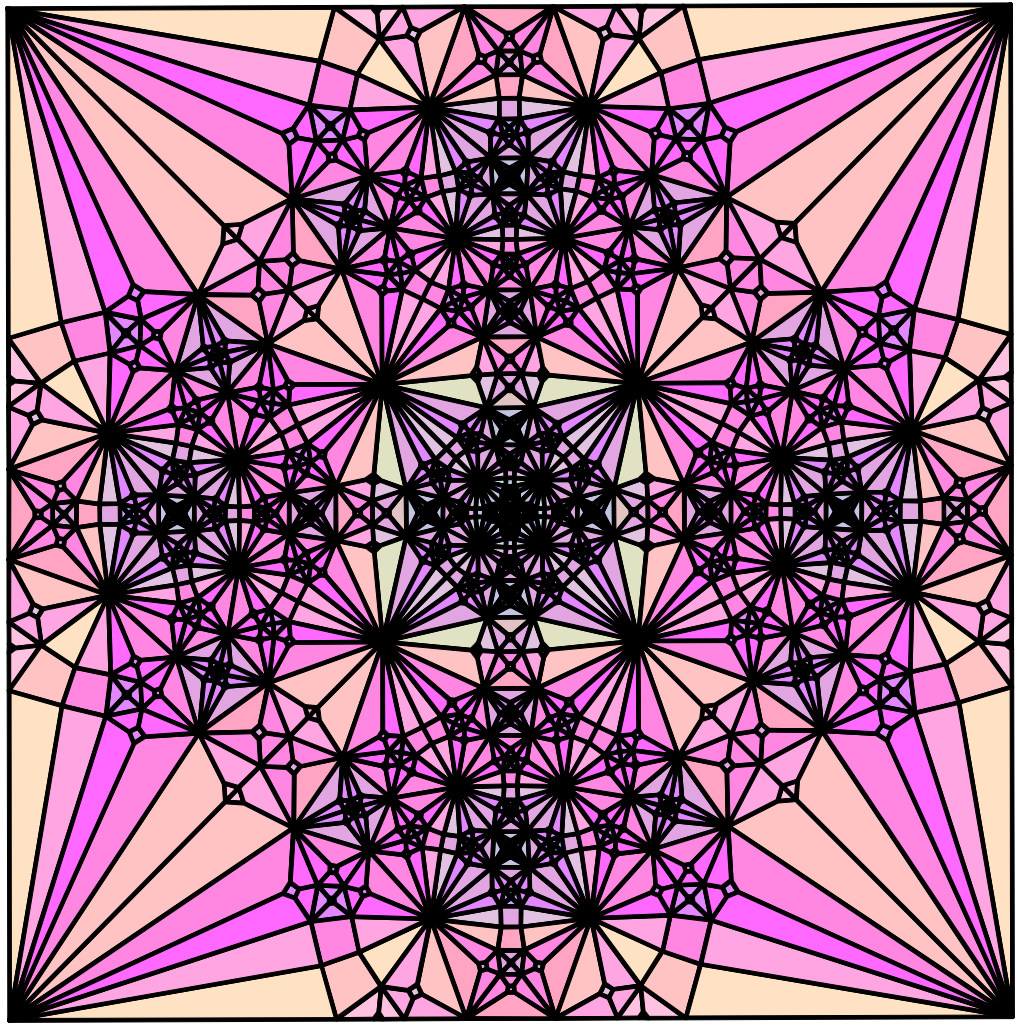}}\caption{A further subdivision of the edge-cycle length 6 cube.}
\label{CircleCubeBig}
\end{center}
\end{figure}

\bibliographystyle{plain}
\bibliography{ManiSubs}

\end{document}